\documentclass[a4paper,reqno,9pt,oneside]{amsart}
\usepackage{amsmath,geometry}
\usepackage{color}

\numberwithin{equation}{section}

\newtheorem{thm}{Theorem}[section]

\newtheorem{lm}[thm]{Lemma}
\newtheorem{oss}[thm]{Remark}
\newtheorem{cor}[thm]{Corollary}

\begin{document}

\date{\today}

\title[On a quasilinear mean field equation with an exponential nonlinearity]{On a quasilinear mean field equation with an exponential nonlinearity}

\author{Pierpaolo Esposito}
\address{Pierpaolo Esposito, Dipartimento di Matematica e Fisica, Universit\`a degli Studi Roma Tre',\,Largo S. Leonardo Murialdo 1, 00146 Roma, Italy}
\email{esposito@mat.uniroma3.it}

\author{Fabrizio Morlando}
\address{Fabrizio Morlando, Dipartimento di Matematica e Fisica, Universit\`a degli Studi Roma Tre,\,Largo S. Leonardo Murialdo 1, 00146 Roma, Italy}
\email{morlando@mat.uniroma3.it}
\thanks{The first author was partially supported by Gruppo Nazionale per l'Analisi Matematica, la Probabilit\'a  e le loro Applicazioni (GNAMPA) of the Istituto Nazionale di Alta Matematica (INdAM), by the Prin project ``Critical Point Theory and Perturbative Methods for Nonlinear Differential Equations" and the Firb-Ideas project ``Analysis and Beyond".}

\begin{abstract} The mean field equation involving the $N-$Laplace operator and an exponential nonlinearity is considered in dimension $N\geq 2$ on bounded domains with homogenoeus Dirichlet boundary condition. By a detailed asymptotic analysis we derive a quantization property in the non-compact case, yielding to the compactness of the solutions set in the so-called non-resonant regime. In such a regime, an existence result is then provided by a variational approach.
\end{abstract}

\maketitle

\section{Introduction}\label{Sec1}
\noindent We are concerned with the following quasilinear mean field equation
\begin{equation}\label{E1}
\left\{ \begin{array}{ll}
          -\Delta_{N} u=\lambda \frac{V e^{u}}{\int_{\Omega}V e^{u}dx}& \mbox{in}\ \Omega\\
          u=0& \mbox{on}\ \partial\Omega
        \end{array} \right.
\end{equation}
on a smooth bounded domain $\Omega\subset \mathbb{R}^{N}$, $N \geq 2$, where $\Delta_N u=\hbox{div}\,(|\nabla u|^{N-2}\nabla u)$ denotes the $N-$Laplace operator, $V$ is a smooth nonnegative function and $\lambda \in \mathbb{R}$. In the sequel, \eqref{E1} will be referred to as the $N$-mean field equation.

\medskip \noindent In terms of $\lambda$ or $\rho=\frac{\lambda}{\int Ve^u}$, the planar case $N=2$ on Euclidean domains or on closed Riemannian surfaces has strongly attracted the mathematical interest, as it arises in conformal geometry \cite{CGY,CY,KW}, in statistical mechanics \cite{CLMP1,CLMP2,ChKi,Kie}, in the study of turbulent Euler flows \cite{DEM,Stuart} and in connection with self-dual condensates for some Chern-Simons-Higgs model \cite{CLW,DEFM,DJLW,EsFi,LiWa,LinYan,NoTa}. 

\medskip \noindent For $N=2$ Br\'{e}zis and Merle \cite{12} initiated the study of the asymptotic behavior for solutions of \eqref{E1} by providing a concentration-compactness result in $\Omega$ without requiring any boundary condition. A quantization property for concentration masses has been later given in \cite{32}, and a very refined asymptotic description has been achieved in \cite{cl1,31}. A first natural question concerns the validity of a similar asymptotic behavior in the quasilinear case $N>2$, where the nonlinearity of the differential operator creates an additional difficulty. The only available result is a concentration-compactness result \cite{3,i}, which provides a too weak compactness property towards existence issues for \eqref{E1}. 
Since  a complete classification for the limiting problem
\begin{equation} \label{1934}
\left\{ \begin{array}{ll}
          -\Delta_{N} U=e^U \ \mbox{in}\ \mathbb{R}^{N}\\
          \int_{\mathbb{R}^{N}}e^U<\infty
        \end{array} \right.
\end{equation}
is not available for $N>2$ (except for extremals of the corresponding Moser-Trudinger's inequality \cite{051,x}) as opposite to the case $N=2$ \cite{15}, the starting point of Li-Shafrir's analysis \cite{32} fails and a general quantization property is completely missing. 
Under a ``mild" control on the boundary values of $u$, Y.Y.Li and independently Wolanski have proposed for $N=2$ an alternative approach based on Pohozaev identities, successfully applied also in other contexts \cite{030,08,07}. The typical assumption on $V$ is the following:
\begin{equation}\label{E15}
\frac{1}{C_0}\leq V(x)\leq C_0 \hbox{ and }|\nabla V(x)| \leq C_0 \qquad \forall\,x \in \Omega
\end{equation}
for some $C_0>0$.

\medskip \noindent Pushing the analysis of \cite{3,i} up to the boundary and making use of the above approach, our first main result is the following:
\begin{thm}\label{T1}
Let $u_{k}\in C^{1,\alpha}(\overline{\Omega})$, $\alpha \in (0,1)$, be a sequence of weak solutions to
\begin{equation}\label{E2}
-\Delta_N u_k=V_k e^{u_k}\qquad \mbox{in}\ \Omega,
\end{equation}
where $V_k$ satisfies \eqref{E15} for all $k\in\mathbb{N}$. Assume that
\begin{equation}\label{E21913}
\sup_{ k \in \mathbb{N}} \int_\Omega e^{u_k}<+\infty
\end{equation}
and
$$\mbox{osc}_{\partial\Omega} u_{k}=\sup_{\partial\Omega}u_{k}-\inf_{\partial\Omega}u_{k}\leq M$$
for some $M \in \mathbb{R}$. Then, up to a subsequence, $u_{k}$ verifies one of the following alternatives: either\\
(i) $u_{k}$ is uniformly bounded in $L_{loc}^{\infty}(\Omega)$\\
or\\
(ii) $u_{k}\rightarrow-\infty$ as $k\rightarrow+\infty$ uniformly in $L_{loc}^{\infty}(\Omega)$\\
or\\
(iii) there exists a finite, non-empty set $S=\{p_{1},...,p_{m}\}\subset\Omega$ such that $u_{k}\rightarrow-\infty$ uniformly in $L_{loc}^{\infty}(\Omega\setminus S)$ and \begin{equation} \label{1919}
V_{k}e^{u_{k}}\rightharpoonup c_N \sum_{i=1}^{m} \delta_{p_i}
\end{equation}
weakly in the sense of measures in $\Omega$ as $k\rightarrow+\infty$, where $c_N=N\big(\frac{N^{2}}{N-1}\big)^{N-1} \omega_{N}$ with $\omega_{N}=|B_1(0)|$. In addition, if $\mbox{osc}_{\partial\Omega} u_{k}=0$ for all $k$, alternatives (i)-(iii) do hold in $\overline{\Omega}$, with $S \subset \Omega$ in case (iii).
\end{thm}
\noindent Without an uniform control on the oscillation of $u_k$ on $\partial \Omega$, in general the concentration mass $\alpha_i$ in \eqref{1919} at each $p_i$, $i=1,\dots,m$, just satisfies $\alpha_{i}\geq N^N \omega_{N}$, see \cite{3,i} for details. Moreover, the assumption $\mbox{osc}_{\partial\Omega} u_{k}=0$ is used here to rule out boundary blow-up. For strictly convex domains, one could simply use the moving-plane method to exclude maximum points of $u_k$ near $\partial \Omega$ as in \cite{i}. For $N=2$ this extra assumption can be removed by using the Kelvin transform to take care of non-convex domains, see \cite{38,ii}. Although $N-$harmonic functions in $\mathbb{R}^N$ are invariant under Kelvin transform, such a property does not carry over to \eqref{E2} due to the nonlinearity of $-\Delta_{N}$. 
To overcome such a difficulty, we still make use of the Pohozaev identity near boundary points, to exclude the boundary blow-up as in \cite{marpe,robwei}.

\medskip \noindent Problem \eqref{1934} has a $(N+1)-$dimensional family of explicit solutions $U_{\epsilon,p}(x)=U(\frac{x-p}{\epsilon})-N \log \epsilon$, $\epsilon>0$ and $p \in \mathbb{R}^N$, where
\begin{equation}\label{E16}
U(x)=\log\frac{F_N}{(1+|x|^{\frac{N}{N-1}})^N},\quad x\in \mathbb{R}^{N},
\end{equation}
with $F_{N}=N\big(\frac{N^{2}}{N-1}\big)^{N-1}$. As $\epsilon \to 0^+$, a description of the blow-up behavior at $p$ is well illustrated by $U_{\epsilon,p}$. Since
$$\int_{\mathbb{R}^N} e^{U_{\epsilon,p}}=c_N,$$
in analogy with Li-Shafrir's result it is expected that the concentration mass $\alpha_i$ in \eqref{1919} at each $p_i$, $i=1,...,m$, should be an integer multiple of $c_{N}$. The additional assumption $\sup_k \mbox{osc}_{\partial\Omega} u_{k}<+\infty$ allows us to prove that all the blow-up points $p_i$, $i=1,\dots,m$, are ``simple" in the sense $\alpha_i=c_N$.

\medskip \noindent Concerning the $N$-mean field equation \eqref{E1}, as a simple consequence of Theorem \ref{T1} we deduce the following crucial compactness property:
\begin{cor}\label{T11}
Let $\Lambda\subset [0,+\infty)\setminus c_N \mathbb{N}$ be a compact set. Then, there exists a constant $C>0$ such that
$\|u\|_\infty \leq C$ does hold for all $\lambda\in\Lambda$, all weak solution $u\in C^{1,\alpha}(\overline{\Omega})$, $\alpha \in (0,1)$, of \eqref{E1} and all $V$ satisfying \eqref{E15}. 
\end{cor}
\noindent In the sequel, we will refer to the case $\lambda\neq c_{N}\mathbb{N}$ as the \emph{non-resonant regime}. Existence issues can be attacked by variational methods: solutions of  \eqref{E1} can be found as critical points of
\begin{equation}\label{E27}
J_{\lambda}(u)=\frac{1}{N}\int_{\Omega}|\nabla u|^{N}-\lambda\log\bigg(\int_{\Omega}Ve^{u}\bigg),\ u\in W_{0}^{1,N}(\Omega).
\end{equation}
The Moser-Trudinger inequality \cite{zz} guarantees that the functional $J_{\lambda}$
is well-defined and $C^{1}$-Fr\'{e}chet differentiable on $W_{0}^{1,N}(\Omega)$ for any $\lambda\in\mathbb{R}$. Moreover, if $\lambda<c_{N}$ the functional $J_{\lambda}$ is coercive and then attains the global minimum. For $\lambda=c_{N}$ $J_{\lambda}$ still has a lower bound but is not coercive anymore: in general, in the resonant regime $\lambda \in c_N \mathbb{N}$ existence issues are very delicate. When $\lambda>c_{N}$ the functional $J_{\lambda}$ is unbounded both from below and from above, and critical points have to be found among saddle points. Moreover, the \emph{Palais-Smale condition} for $J_{\lambda}$ is not globally available, see \cite{aha}, but holds only for bounded sequences in $W^{1,N}_0(\Omega)$.

\medskip \noindent The second main result is the following:
\begin{thm} \label{thm0954}
Assume that the space of formal barycenters $\mathfrak{B}_{m}(\overline{\Omega})$ of $\overline{\Omega}$ with order $m \geq 1$ is non contractible. Then equation \eqref{E1} has a solution in $C^{1,\alpha}(\overline{\Omega})$, $\alpha \in (0,1)$, for all $\lambda\in\big(c_{N}m,c_{N}(m+1)\big)$.
\end{thm}
\noindent For mean-field equations, such a variational approach has been introduced in \cite{DJLW1} and fully exploited later by Djadli and Malchiodi \cite{0120} in their study of constant $Q$-curvature metrics on four manifolds. It has revelead to be very powerful in many contexts, see for example \cite{007,BBB,012,013} and refences therein. Alternative approaches are available: the computation of the corresponding Leray-Schauder degree \cite{cl1,010}, based on a very refined asymptotic analysis of blow-up solutions; perturbative constructions of Lyapunov-Schimdt in the almost resonant regime \cite{BP,010,DEFM,DEM,DKM,EsFi,EGP,LinYan}. For our problem a refined asymptotic analysis for blow-up solutions is still missing, and perturbation arguments are very difficult due to the nonlinearity of $\Delta_N$. A variational approach is the only reasonable way to attack existence issues, and in this way the analytic problem is reduced to a topological one concerning the non-contractibility of a model space, the so-called \emph{space of formal barycenters}, characterizing the very low sublevels of $J_{\lambda}$. We refer to Section \ref{Sec5} for a definition of $\mathfrak{B}_{m}(\overline{\Omega})$. To have non-contractibility of $\mathfrak{B}_{m}(\overline{\Omega})$ for domains $\Omega$ homotopically equivalent to a finite simplicial complex,  a sufficient condition is the non-triviality of the $\mathbb{Z}$-homology, see \cite{kal}. Let us emphasize that the variational approach produces solutions a.e. $\lambda\in\big(c_{N}m,c_{N}(m+1)\big)$, $m \geq 1$, and Corollary \ref{T11} is crucial to get the validity of Theorem \ref{thm0954} for all $\lambda$ in such a range.

\medskip \noindent The paper is organized as follows. In Section \ref{Sec3} we show how to push the concentration-compactness analysis \cite{3,i} up to the boundary, by discussing boundary blow-up and mass quantization. Section \ref{Sec5} is devoted to Theorem \ref{thm0954} and some comments concerning $\mathfrak{B}_{m}(\overline{\Omega})$. In the appendix, we collect some basic results that will be used frequently throughout the paper.

%\medskip\noindent
%{\bf Acknowledgments:}

\section{Concentration-Compactness analysis}\label{Sec3}
\noindent Even though representation formulas are not available for $\Delta_N$, the Br\'{e}zis-Merle's inequality \cite{12} can be extended to $N>2$ by different means:
\begin{lm} \cite{3,i} \label{T4}
Let $u \in C^{1,\alpha}(\overline{\Omega})$ be a weak solution of 
$$-\Delta_{N} u=f\ \quad \mbox{in}\ \Omega$$
with $f\in L^{1}(\Omega)$. Let $\varphi$ be a $N$-harmonic function in $\Omega$ with $\varphi=u$ on $\partial\Omega$. Then, for every $\alpha\in(0,\alpha_{N})$ there exists a constant $C=C(\alpha,|\Omega|)$ such that
\begin{equation}\label{E11}
\displaystyle\int_{\Omega}\exp\bigg[\frac{\alpha |u(x)-\varphi(x)|}{\|f\|_{L^{1}}^{\frac{1}{N-1}}}\bigg]\leq C,
\end{equation}
where $\alpha_N=(N^N d_{N} \omega_{N})^{\frac{1}{N-1}} $ and
$$d_{N}=\inf_{X\neq Y\in \mathbb{R}^{N}}\frac{<|X|^{N-2}X-|Y|^{N-2}Y,X-Y>}{|X-Y|^{N}}\,>0.$$
In addition, if $u=0$ on $\partial \Omega$ inequality \eqref{E11} does hold with $\alpha_N=(N^N \omega_{N})^{\frac{1}{N-1}}$.
\end{lm}
%\noindent Moreover, let us observe that
%\begin{oss}\label{R1}
%Let $u$ be a entropy solution of 
%$$\left\{ \begin{array}{ll}-\Delta_{N} u=V e^u & \mbox{in }\Omega \\ u=0 & \hbox{on }\partial \Omega \end{array} \right.$$
%with $V \in L^p(\Omega)$ and $V e^u \in L^1(\Omega)$, for some $1<p\leq +\infty$. In \cite{3} it is proved that $u \in L^\infty(\Omega)$. The choice $M >\|u\|_{L^{\infty}(\Omega)}$ in %\eqref{entropysol} gives that $u \in W^{1,N}_0(\Omega)$, and by Theorem \ref{T3} $u \in C^{1,\alpha}(\Omega)$, for some $\alpha \in (0,1)$.
%\end{oss}
%\noindent The notion of entropy solution, which has been introduced in \cite{7}, is reproduced for convenience in Definition \ref{entropy}. The validity of Lemma \ref{T4} and Remark %\ref{R1} in the general setting described by Definition \ref{entropy} will be crucial to exclude the boundary blow-up in Theorem \ref{T1}. Hereafter, we restrict the attention to the case of %``classical'' solutions, which a-posteriori won't be restrictive in view of Remark \ref{R1}. 
\noindent Under some smallness uniform condition on the nonlinear term, a-priori estimates hold true as follows:
\begin{lm}\label{T6}
Let $u_{k}\in C^{1,\alpha}(\overline{\Omega})$, $\alpha \in (0,1)$, be a sequence of weak solutions to \eqref{E2}, where $V_k$ satisfies \eqref{E15} for all $k\in\mathbb{N}$. Assume that 
\begin{equation}\label{E12}
\sup_k \int_{\Omega\cap B_{4R}}V_{k}e^{u_{k}}<N^N d_{N} \omega_{N}
\end{equation}
does hold for some $R>0$, and $u_k$ satisfies $u_k=c_k$ on $\partial \Omega \cap \overline{B_{4R}}$, $u_k \geq c_k$ in $\Omega \cap B_{4R}$
for $c_k \in \mathbb{R}$ if $\partial \Omega \cap \overline{B_{4R}} \neq \emptyset$. Then
\begin{equation} \label{1540}
\sup_k \|u^{+}_{k}\|_{L^{\infty}(\Omega\cap B_{R})}<+\infty.
\end{equation}
\end{lm}
\begin{proof}
Let $\varphi_k$ be the $N-$harmonic function in $\Omega \cap B_{4R}$ so that $\varphi_k=u_k$ on $\partial(\Omega \cap B_{4R})$. Choosing 
$$\alpha \in \left( (\sup_k \int_{\Omega \cap B_{4R}} V_k e^{u_k})^{\frac{1}{N-1}},\alpha_N\right)$$ 
in view of \eqref{E12}, by Lemma \ref{T4} we get that $e^{|u_k-\varphi_k|}$ is uniformly bounded in $L^q(\Omega \cap B_{4R})$, for some $q>1$. Since $V_k \geq 0$, by the weak comparison principle we get that $c_k\leq \varphi_k \leq u_k$ in $\Omega \cap B_{4R}$. Since $\varphi_k=c_k$ on $\partial \Omega \cap \overline{B_{4R}}$ and
\begin{equation}\label{1507}
\sup_k \|\varphi_k^+\|_{L^N (\Omega\cap B_{4R})}\leq \sup_k \|u^{+}_{k}\|_{L^{N}(\Omega\cap B_{4R})}<+\infty
\end{equation}
in view of \eqref{E15} and \eqref{E12}, by Theorem \ref{T5} we get that $\varphi_k \leq C_0$ in $\Omega \cap B_{2R}$ uniformly in $k$, for some $C_0$. Since $e^{u_k}\leq e^{C_0} e^{|u_k-\varphi_k|}$, we get that $e^{u_k}$ is uniformly bounded in $L^q(\Omega \cap B_{2R})$. Since $q>1$, by Theorem \ref{T5} we deduce the validity of \eqref{1540} in view of \eqref{1507}.
\end{proof}
\noindent We can now prove our first main result:

\medskip \noindent \emph{Proof (of Theorem \ref{T1}).}\\
First of all, by \eqref{E15} for $V_k$ and \eqref{E21913} we deduce that $V_{k}e^{u_{k}}$ is uniformly bounded in $L^{1}(\Omega)$. Up to a subsequence, by the Prokhorov Theorem we can assume that $V_{k}e^{u_{k}}\rightharpoonup\mu\in\mathcal{M}^{+}(\overline{\Omega})$ as $k\rightarrow+\infty$ in the sense of measures in $\overline{\Omega}$, i.e.
$$\int_\Omega V_k e^{u_k} \varphi \to \int_\Omega \varphi d\mu \hbox{ as }k \to +\infty \qquad \forall \ \varphi \in C(\overline{\Omega}).$$
A point $p\in\overline{\Omega}$ is said a \emph{regular point} for $\mu$ if $\mu(\{p \})<N^N \omega_{N}$, and let us denote the set of non-regular points as:
$$\Sigma=\{p \in \overline{\Omega}:\, \mu(\{p \})\geq N^N \omega_{N} \}.$$
Since $\mu$ is a bounded measure, it follows that $\Sigma$ is a finite set. We complete the argument through the following five steps.

\medskip \noindent \textbf{Step 1}  Letting $$S=\big\{p \in\overline{\Omega}:\ \limsup_{k\rightarrow+\infty}\sup_{\Omega \cap B_{R}(p)}u_{k}=+\infty\ \forall R>0\big\},$$ 
there holds $S\cap\Omega=\Sigma\cap\Omega$ ($S=\Sigma$ if $\emph{osc}_{\partial\Omega}u_{k}=0$ for all $k$). 

\medskip \noindent Letting $p_0 \in S$, assume that $p_0 \in \Omega$ or $u_k=c_k$ on $\partial\Omega$ for some $c_k \in \mathbb{R}$. In the latter case, notice that $u_k \geq c_k$ in $\Omega$ in view of the weak comparison principle. Setting
$$\Sigma'=\bigg\{p \in\overline{\Omega}:\ \mu(\{p \})\geq N^N d_{N}\omega_{N} \bigg\},$$
by Lemma \ref{T6} we know that $p_0 \in\Sigma'$. Indeed, if $p_{0} \notin \Sigma'$, then \eqref{E12} would hold for some $R>0$ small, and then by Lemma \ref{T6} it would follow that $u_{k}$ is uniformly bounded from above in $\Omega\cap B_{R}(p_{0})$, contradicting $p_{0} \in S$. To show that $p_{0}\in\Sigma$, the key point is to recover a good control of $u_{k}$ on $\partial\big(\Omega\cap B_{R}(p_{0})\big),$ for some $R>0$, in order to drop $d_{N}$. Assume that $p_{0}\notin\Sigma$, in such a way that 
\begin{equation} \label{1923}
\sup_k \int_{\Omega\cap B_{2R}(p_{0})}V_{k}e^{u_{k}}< N^N \omega_{N}
\end{equation} for some $R>0$ small. Since $\Sigma'$ is a finite set, up to take $R$ smaller, let us assume that $\partial\big(\Omega\cap B_{2R}(p_{0})\big)\cap\Sigma' \subset \{p_{0}\},$ and then by compactness we have that
\begin{equation}\label{E13qq}
u_{k}\leq M\quad \mbox{in}\ \partial\big(\Omega\cap B_{2R}(p_{0})\big)\setminus \ B_{R}(p_{0})
\end{equation}
in view of $S \cap \Omega \subset \Sigma' \cap \Omega$ and $S \subset \Sigma'$ if $\emph{osc}_{\partial\Omega}u_{k}=0$ for all $k$. If $p_{0}\in\Omega$, we can also assume that $\overline{B_{2R}(p_{0})}\subset \Omega.$ If $p_{0}\in\partial\Omega$, $u_{k}=c_{k}$ on $\partial\Omega$ yields to $c_{k}\leq M$ in view of \eqref{E13qq}. In both cases, we have shown that \eqref{E13qq} does hold in the stronger way:
\begin{equation}\label{E14}
u_{k}\leq M\quad \mbox{in}\ \partial\big(\Omega\cap B_{2R}(p_{0})\big).
\end{equation}
Letting $w_{k} \in W^{1,N}_0(\Omega\cap B_{2R}(p_{0}))$ be the weak solution of
$$
\left\{ \begin{array}{ll}
-\Delta_{N} w_{k}=V_{k}e^{u_{k}}& \mbox{in}\ \Omega\cap B_{2R}(p_{0})\\
w_{k}=0& \mbox{on}\ \partial\big(\Omega\cap B_{2R}(p_{0})\big),\
        \end{array} \right.
\vspace{0,2cm}$$
by \eqref{E14} and the weak comparison principle we get that 
$$u_{k}\leq w_{k}+M\quad \mbox{in}\ \Omega\cap B_{2R}(p_{0}).$$
Applying Lemma \ref{T4} to $w_k$ in view of \eqref{1923}, it follows that 
$$\int_{\Omega\cap B_{2R}(p_{0})}e^{q u_{k}}\leq e^{q M}\int_{\Omega\cap B_{2R}(p_{0})}e^{qw_{k}}\leq C$$
for all $k$, for some $q>1$ and $C>0$. In particular, $u_{k}^{+}$ is uniformly bounded in $L^{N}\big(\Omega\cap B_{2R}(p_{0})\big)$ and $V_{k}e^{u_{k}}$ is uniformly bounded in $L^q \big(\Omega\cap B_{2R}(p_{0})\big)$. By Theorem \ref{T5} it follows that $u_{k}$ is uniformly bounded from above in $\Omega\cap B_{R}(p_{0}),$ in contradiction with $p_{0}\notin S.$ So, we have shown that $p_{0}\in\Sigma$, which yields to $S \cap \Omega \subset \Sigma\cap\Omega$ and $S \subset\Sigma$ if $\emph{osc}_{\partial\Omega}u_{k}=0$ for all $k$.\\ 
Conversely, let $p_{0}\in\Sigma$. If $p_{0}\notin S$, one could find $R_{0}>0$ so that $u_{k}\leq M$ in $\Omega\cap B_{R_{0}}(p_{0}),$ for some $M\in\mathbb{R}$, yielding to $$\int_{\Omega\cap B_{R}(p_{0})}V_{k}e^{u_{k}}\leq C_0 e^{M} R^{N},\ R\leq R_{0},$$ in view of \eqref{E15}. In particular, $\mu(\{p_{0}\})=0$, contradicting $p_{0}\in\Sigma$. Hence $\Sigma\subset S$, and the proof of Step 1 is complete.

\medskip \noindent \textbf{Step 2}
$S \cap \Omega=\emptyset$ ($S=\emptyset$) implies the validity of alternative (i) or (ii) in $\Omega$ (in $\overline{\Omega}$ if $\emph{osc}_{\partial\Omega}u_{k}=0$ for all $k$). 

\medskip \noindent Since $u_{k}$ is uniformly bounded from above in $L^{\infty}_{loc}(\Omega)$, then either $u_{k}$ is uniformly bounded in $L^{\infty}_{loc}(\Omega)$ or there exists, up to a subsequence, a compact set $K\subset\Omega$ so that $\min_{K}u_{k}\rightarrow-\infty$ as $k\rightarrow+\infty$. The set $\Omega_{\delta}=\{x\in\Omega: \mbox{dist}(x,\partial\Omega)\geq \delta\}$ is a compact connected set so that $K\subset \Omega_{\delta}$, for $\delta>0$ small. Since $u_{k}\leq M$ in $\Omega$ for some $M>0$, the function $s_{k}=M-u_{k}$ is a nonnegative weak solution of $-\Delta_{N}s_{k}=-V_{k}e^{u_{k}}$ in $\Omega$. By the Harnack inequality in Theorem \ref{T7}, we have that
$$\max_{\Omega_{\delta}}s_{k}\leq C\big(\min_{\Omega_{\delta}}s_{k}+1 \big)$$ 
in view of
$$\|V_k e^{u_k}\|_{L^\infty(\Omega)} \leq C_0 e^M.$$
In terms of $u_k$, it reads as
$$\max_{\Omega_{\delta}}u_{k}\leq M\big(1-\frac{1}{C}\big)+1+\frac{1}{C}\min_{K}u_{k} \to -\infty$$ 
as $k \to +\infty$ for all $\delta>0$ small, yielding to the validity of alternative (ii) in $\Omega$. Assume in addition that $u_k=c_k$ on $\partial \Omega$ for some $c_k \in \mathbb{R}$. Notice that $c_k \leq u_k \leq M$ in $\Omega$ for all $k$. If alternative (i) does not hold in $\overline{\Omega}$, up to a subsequence, we get that $c_k \to -\infty$. Since $V_k e^{u_k}$ is uniformly bounded in $\Omega$, we apply Corollary \ref{2002} to $s_k=u_k-c_k$, a nonnegative solution of $-\Delta_{N}s_{k}=V_{k}e^{u_{k}}$ with $s_k=0$ on $\partial \Omega$, to get $s_k \leq M'$ in $\Omega$ for some $M' \in \mathbb{R}$. Hence, $u_k \leq M'+c_k \to -\infty$ in $\Omega$ as $k \to +\infty$, yielding to the validity of alternative (ii) in $\overline{\Omega}$. The proof of Step 2 is complete.

\medskip \noindent \textbf{Step 3} $S \cap \Omega \neq \emptyset$ implies the validity of alternative (iii) in $\Omega$ (in $\overline{\Omega}$ if $\emph{osc}_{\partial\Omega}u_{k}=0$ for all $k$) with \eqref{1919} replaced by the property:
\begin{equation} \label{2242}
V_k e^{u_k} \rightharpoonup \sum_{i=1}^m \alpha_i \delta_{p_i}
\end{equation}
weakly in the sense of measures in $\Omega$ (in $\overline{\Omega}$) as $k \to +\infty$, with $\alpha_i \geq N^N \omega_N$ and $S \cap \Omega=\{p_1,\dots,p_m\}$ ($S=\{p_1,\dots,p_m\}$).  

\medskip \noindent Let us first consider the case that $u_{k}$ is uniformly bounded in $L^{\infty}_{loc}(\Omega\setminus S)$. Fix $p_{0}\in S$ and $R>0$ small so that $\overline{B_{R}(p_{0})} \cap S =\{p_{0}\}$. Arguing as in \eqref{E13qq}-\eqref{E14}, we have that $u_{k}\geq m$ on $\partial(\Omega \cap B_{R}(p_{0}))$ for some $m\in\mathbb{R}$. Since $u_k$ is uniformly bounded in $L^\infty_{loc}(\Omega \setminus S)$, by Theorem \ref{T3} it follows that $u_{k}$ is uniformly bounded in $C^{1,\alpha}_{loc}(\overline{\Omega \cap B_R(p_0)} \setminus \{p_0\})$, for some $\alpha\in(0,1)$, and, up to a subsequence and a diagonal process, we can assume that $u_{k}\rightarrow u$ in $C^{1}_{loc}(\overline{\Omega \cap B_R(p_0)} \setminus \{p_0 \} )$ as $k \to +\infty$. By \eqref{E15} on each $V_k$, we can also assume that $V_k \to V$ uniformly in $\Omega$ as $k \to +\infty$. Hence, there holds
\begin{equation} \label{1339} 
V_{k}e^{u_{k}} \rightharpoonup \mu=V e^u \ dx+\alpha_0 \delta_{p_{0}}
\end{equation} 
weakly in the sense of measures in $\overline{\Omega \cap B_R(p_{0})}$ as $k \to +\infty$, where $\alpha_0  \geq N^N \omega_N$. Since
$$\lim_{k\to +\infty} \int_{\Omega \cap B_R(p_0)} V_k e^{u_k}=\int_{\Omega \cap B_R(p_0)} V e^u+\alpha_0>\alpha_0$$
in view of \eqref{1339}, for $k$ large we can find a unique $0<r_k<R$ so that 
\begin{equation} \label{1238} 
\int_{\Omega \cap B_{r_k}(p_0)} V_{k}e^{u_{k}}=\alpha_0.
\end{equation} 
Notice that $r_k \to 0$ as $k \to +\infty$. Indeed, if $r_k \geq \delta>0$ were true along a subsequence, one would reach the contradiction
$$\alpha_0 \geq \int_{\Omega \cap B_\delta(p_0)} V_k e^{u_k} \to \int_{\Omega \cap B_\delta(p_0)} V e^u+\alpha_0>\alpha_0$$
as $k \to +\infty$ in view of \eqref{1339}-\eqref{1238}. Denoting by $\chi_A$ the characteristic function of a set $A$, we have the following crucial property:
$$\chi_{B_{r_k}(p_0)} V_k e^{u_k}  \rightharpoonup \alpha_0 \delta_{p_{0}}$$
weakly in the sense of measures in $\overline{\Omega \cap B_R(p_{0})}$ as $k \to +\infty$, as it easily follows by \eqref{1238} and $\displaystyle \lim_{k \to +\infty} r_k=0$. 

\medskip \noindent We can now specialize the argument to deal with the case $p_0 \in S \cap \Omega$. Assume that $R$ is small so that $\overline{B_{R}(p_{0})} \subset\Omega$. Letting $w_{k} \in W^{1,N}_0(B_R(p_0))$ be the weak solution of
$$\left\{ \begin{array}{ll}
-\Delta_{N} w_{k}=\chi_{B_{r_k}(p_0)} V_{k}e^{u_{k}} & \mbox{in}\ B_{R}(p_{0})\\
w_{k}=0& \mbox{on}\ \partial B_{R}(p_{0}),\
        \end{array} \right. $$
by the weak comparison principle there holds $0\leq w_{k} \leq u_{k}-m$ in $B_{R}(p_{0})$ in view of $0\leq \chi_{B_{r_k}(p_0)}  V_k e^{u_k} \leq V_k e^{u_k}$. Arguing as before, up to a subsequence, by Theorem \ref{T3}  we can assume that $w_k \to w$ in $C^{1}_{loc}(\overline{B_R(p_0)} \setminus \{p_0 \} )$ as $k \to +\infty$, where $w\geq 0$ is a $N-$harmonic and continous function in $B_R(p_0) \setminus \{p_0\}$ which solves
$$-\Delta_N w=\alpha_0 \delta_{p_0 } \quad \hbox{in }B_R(p_0)$$
in a distributional sense. By Theorem \ref{T2} we deduce that 
\begin{equation} \label{1406}
w \geq (N\omega_{N})^{-\frac{1}{N-1}}  \alpha_0^{\frac{1}{N-1}}  \log \frac{1}{|x-p_0|}+C\geq  N \log \frac{1}{|x-p_0|}+C \qquad \hbox{in }B_r(p_0)
\end{equation}
in view of $\alpha_0 \geq N^N \omega_N$, for some $C \in \mathbb{R}$ and $0<r\leq \min\{1,R\}$. Since
$$\int_{B_{R}(p_{0})}e^{w_{k}}\leq e^{-m}\sup_k \int_{\Omega}e^{u_{k}}<+\infty$$
in view of \eqref{E21913}, as $k \to +\infty$ we get that $\int_{B_{R}(p_{0})}e^{w}<+\infty$, in contradiction with \eqref{1406}:
$$\int_{B_{R}(p_{0})}e^w \geq e^C \int_{B_r(p_{0})}\frac{1}{|x-p_{0}|^{N}}=+\infty.$$ 
Since $u_{k}$ is uniformly bounded from above and not from below in $L^{\infty}_{loc}(\Omega\setminus S)$, there exists, up to a subsequence, a compact set $K\subset\Omega \setminus S$ so that $\min_{K} u_{k}\rightarrow-\infty$ as $k\rightarrow+\infty$.  Arguing as in Step 2 by simply replacing $\mbox{dist}(\cdot,\partial\Omega)$ with $\mbox{dist}(\cdot,\partial\Omega\cap S)$, we can show that $u_k \to -\infty$ in $L^\infty_{loc}(\Omega \setminus S)$ as $k \to +\infty$, and \eqref{2242} does hold in $\Omega$ with $\{p_1,\dots,p_m\}=S\cap \Omega$. If in addition $u_k=c_k$ on $\partial \Omega$ for some $c_k \in \mathbb{R}$, we can argue as in the end of Step 2 (by using Theorem \ref{T7} instead of Corollary \ref{2002}) to get that $u_k \to -\infty$ in $L^\infty_{loc}(\overline{\Omega} \setminus S)$ as $k \to +\infty$, yielding to the validity of \eqref{2242} in $\overline{\Omega}$ with $\{p_1,\dots,p_m\}=S$. The proof of Step 3 is complete.

\medskip \noindent To proceed further we make use of Pohozaev identities. Let us emphasize that $u_k \in C^{1,\alpha}(\overline{\Omega})$, $\alpha \in (0,1)$, and the classical Pohozaev identities usually require more regularity. In \cite{18} a self-contained proof  is provided in the quasilinear case, which reads in our case as:
\begin{lm} \label{Po}
Let $\Omega \subseteq\mathbb{R}^{N}$, $N\geq2$, be a smooth bounded domain, $f$ be a locally Lipschitz continuous function and $0\leq V\in C^{1}(\overline{\Omega})$. 
Then, there holds
$$
\int_{\Omega}\big[N\ V+\langle x-y,\nabla V\rangle\big]F(u)=\int_{\partial\Omega}V\ F(u)\langle x-y,\nu\rangle
+|\nabla u|^{N-2}\langle x-y,\nabla u\rangle \partial_{\nu}u-\frac{|\nabla u|^{N}}{N}\langle x-y,\nu\rangle$$
for all weak solution $u\in C^{1,\alpha}(\overline{\Omega})$, $\alpha \in (0,1)$, of $-\Delta_N u=V f(u)$ in $\Omega$ and all $y\in\mathbb{R}^{N}$, where $F(t)=\displaystyle\int_{-\infty}^{t}f(s)ds$ and $\nu$ is the unit outward normal vector at $\partial \Omega$.
\end{lm}
\noindent Thanks to Lemma \ref{Po}, in the next two Steps we can now describe the interior blow-up phenomenon and exclude the occurence of boundary blow-up:

\medskip \noindent \textbf{Step 4} If $\emph{osc}_{\partial\Omega}u_{k}\leq M$ for some $M \in \mathbb{R}$, then $\alpha_i=c_N$ for all $p_i \in S \cap \Omega$. 

\medskip \noindent Since $0\leq u_{k}-\inf_{\partial\Omega}u_{k}\leq M$ on $\partial\Omega$, we have that $s_{k}=u_{k}-\inf_{\partial\Omega}u_{k}\geq 0$ satisfies
$$
\left\{ \begin{array}{ll}
-\Delta_{N} s_{k}= W_k e^{s_k}& \mbox{in}\ \Omega\\
0\leq s_{k}\leq M& \mbox{on}\ \partial\Omega,\
        \end{array} \right.$$
where $W_{k}=V_{k}e^{\inf_{\partial\Omega}u_{k}}$. Letting now $\varphi_{k}$ be the $N-$harmonic function in $\Omega$ with $\varphi_{k}=s_{k}$ on $\partial\Omega$, by the weak comparison principle we have that $0\leq \varphi_{k}\leq M$ in $\Omega$. Since $\sup_k \int_\Omega W_k e^{s_k}<+\infty$ and $e^{\gamma s}\geq\delta s^{N}$ for all $s\geq 0$, for some $\delta>0$, by Lemma \ref{T4} we deduce that $s_k- \varphi_k$ and then $s_k$ are uniformly bounded in $L^N(\Omega)$.
Since $W_k e^{s_k}=V_k e^{u_k}$ is uniformly bounded in $L_{loc}^{\infty}(\overline{\Omega} \setminus S)$, by Theorem \ref{T3} it follows as in Step 3 that, up to a subsequence, $s_{k} \rightarrow s$ in $C^{1}_{loc}(\Omega \setminus S )$. Fix $p_0 \in S\cap \Omega$ and take $R_0>0$ small so that $B=B_{R_0}(p_0)\subset\subset\Omega$ and $\overline{B} \cap S=\{p_0\}$. The limiting function $s \geq0$ is a $N$-harmonic and continuous function in $B \setminus\{p_0\}$ which solves
$$-\Delta_N s=\alpha_0 \delta_{p_0} \quad \hbox{in }B,$$
where $\alpha_0 \geq N^N \omega_N$. By Theorem \ref{T2} we have that $s=\alpha_0^{\frac{1}{N-1}} \Gamma(|x-p_0|)+H$, where $H \in L^\infty_{loc}(B)$ does satisfy
\begin{equation}\label{E17}
\lim_{x\to p_0}|x-p_0||\nabla H(x)|=0.
\end{equation}
Applying the Pohozaev identity to $s_{k}$ on $B_R(p_0)$, $0<R\leq R_0$, with $y=p_0$, we get that
$$\int_{B_R(p_0)} \big[NW_{k}+\langle x-p_0,\nabla W_{k}\rangle\big]e^{s_{k}}=R \int_{\partial B_R(p_0)}\left[W_{k}e^{s_{k}}+|\nabla s_{k}|^{N-2}(\partial_{\nu}s_{k})^{2}-\frac{|\nabla s_{k}|^{N}}{N}\right].$$
Since $S \cap \Omega \neq \emptyset$ and $V_{k}e^{u_{k}}=W_{k}e^{s_{k}}$,
by Step 3 we get that $\int_{\partial B_R(p_0)}W_{k}e^{s_{k}} \to 0$ and
$$\int_{B_R(p_0)}\big[NW_{k}+\langle x-p_0,\nabla W_{k}\rangle\big] e^{s_{k}}=N \int_{B_R(p_0)}V_{k}e^{u_{k}}+O\left(\int_{B_R(p_0)}|x-p_0| V_{k}e^{u_{k}}\right) \to N\alpha_0$$
as $k\to +\infty$. Letting $k\to \infty$ we get that
\begin{eqnarray*}
N \alpha_0&=&R \int_{\partial B_R(p_0)} |\nabla H-(\frac{\alpha_0}{N\omega_{N}})^{\frac{1}{N-1}} \frac{x-p_0}{|x-p_0|^{2}}|^{N-2}[\partial_{\nu}H-(\frac{\alpha_0}{N\omega_{N}})^{\frac{1}{N-1}}\frac{1}{|x-p_0|}]^{2}\\
&&-\frac{R}{N}  \int_{\partial B_R(p_0)} |\nabla H-(\frac{\alpha_0}{N\omega_{N}})^{\frac{1}{N-1}}\frac{x-p_0}{|x-p_0|^{2}}|^{N}\\
&=& R \frac{N-1}{N} \int_{\partial B_R(p_0)} \left[ (\frac{\alpha_0}{N\omega_{N}})^{\frac{2}{N-1}}\frac{1}{|x-p_0|^{2}}+O(\frac{1}{|x-p_0|}|\nabla H|+|\nabla H |^{2})\right]^{\frac{N}{2}}\\
&=& R \frac{N-1}{N}\int_{\partial B_R(p_0)} (\frac{\alpha_0}{N\omega_{N}})^{\frac{N}{N-1}} \frac{1}{|x-p_0|^{N}} \left[1+O(|x-p_0||\nabla H|+|x-p_0|^{2}|\nabla H|^{2})\right]
\end{eqnarray*}
in view of $s_{k}\to s=\alpha_0^{\frac{1}{N-1}} \Gamma(|x-p_0|)+H$ in $C^{1}_{loc}(\overline{B}\setminus\{p_0\})$ as $k\rightarrow+\infty$. Letting $R \to 0$ we get that
$$N \alpha_0=\frac{N-1}{N} (\frac{\alpha_0}{N\omega_{N}})^{\frac{N}{N-1}} N\omega_{N},
$$
in view of \eqref{E17}. Therefore, there holds 
$$\alpha_0 =N\big(\frac{N^{2}}{N-1}\big)^{N-1}\omega_{N}=c_{N}$$
for all $p_0 \in S \cap \Omega$, and the proof of Step 4 is complete.

\medskip \noindent \textbf{Step 5} If $\emph{osc}_{\partial\Omega}u_{k}=0$ for all $k$, then $S \subset \Omega$.

\medskip \noindent Assume now that $u_k=c_k$ on $\partial \Omega$. Since by the weak comparison principle $c_k \leq u_k$ in $\Omega$ for all $k$, the function $s_{k}=u_{k}-c_{k}$ is a nonnegative weak solution of
$$\left\{ \begin{array}{ll}
-\Delta_{N} s_{k}=W_{k}e^{s_{k}}& \mbox{in}\ \Omega\\
s_{k}=0& \mbox{on}\ \partial\Omega,\
        \end{array} \right.$$
where $W_{k}=V_{k}e^{c_{k}}$. Since $W_k e^{s_k}= V_k e^{u_k}$ is uniformly bounded in $L^1(\Omega)$, by Lemma \ref{T4} we have that $s_{k}$ is uniformly bounded in $L^{N}(\Omega)$. Since $W_k e^{s_k}=V_k e^{u_k}$ is uniformly bounded in $L_{loc}^{\infty}(\overline{\Omega} \setminus S)$, arguing as in Step 3, by Theorem \ref{T3} it follows that $s_k$ is uniformly bounded in $C^{1,\alpha}_{loc}(\overline{\Omega} \setminus S )$, $\alpha \in (0,1)$, and, up to a subsequence, $s_{k} \rightarrow s$ in $C^{1}_{loc}(\overline{\Omega} \setminus S )$. We claim that $ s \in C^1(\overline{\Omega})$.\\
If $c_k \to -\infty$, we have that $s \in C^{1}_{loc}(\overline{\Omega}\setminus S )$ is a nonnegative $N$-harmonic function in $\Omega \setminus S$ with $s=0$ on $\partial \Omega \setminus S$. By Theorem \ref{T7} we deduce that $s=0$ in $\Omega$, and then $s \in C^1(\overline{\Omega})$. Up to a subsequence, we can now assume that $c_k \to c \in \mathbb{R}$ as $k \to +\infty$ and $S=\{p_1,\dots,p_m\} \subset \partial \Omega$ in view of Step 3. By \cite{8,9} $s \in W^{1,q}_0(\Omega)$ for all $q<N$ and is a distributional solution of
\begin{equation} \label{1417}
\left\{ \begin{array}{ll}
-\Delta_{N} s=W e^s& \mbox{in}\ \Omega\\
s=0& \mbox{on}\ \partial\Omega\
        \end{array} \right.
\end{equation}
(referred to as SOLA, Solution Obtained as Limit of Approximations), where $W=V e^c$ and $We^s \in L^1(\Omega)$. By considering different $L^1-$approximations or even $L^1-$weak approximations of $We^s \in L^1(\Omega)$ one always get the same limiting SOLA \cite{Dall}, which is then unique in the sense explained right now. Unfortunately, the sequence $W_k e^{s_k}$ does not converge $L^1-$weak to $We^s$  as $k \to +\infty$ since it keeps track that some mass is concentrating near the boundary points $p_1,\dots,p_m$. Given $p=p_i \in S$ and $\alpha=\alpha_i$, arguing as in \eqref{1238} we can find a radius $r_k \to 0$ as $k \to +\infty$ so that
\begin{equation} \label{0939}
\int_{\Omega \cap B_{r_k}(p)} W_{k}e^{s_{k}}=\alpha.
\end{equation}
Let $w_{k} \in W^{1,N}_0(\Omega \cap B_R(p))$ be the weak solution of
$$\left\{ \begin{array}{ll}
-\Delta_{N} w_{k}=\chi_{\Omega \cap B_{r_k}(p)} W_{k}e^{s_{k}} & \mbox{in }\Omega \cap  B_{R}(p)\\
w_{k}=0& \mbox{on}\  \partial(\Omega \cap  B_{R}(p)),
        \end{array} \right. $$
where $R<\frac{1}{2} \hbox{dist}\ (p,S \setminus \{p\})$. Arguing as in Step 3, up to a subsequence, we have that $w_k \to w$ in $C^{1}_{loc}(\overline{\Omega \cap B_R(p)} \setminus \{p \} )$ as $k \to +\infty$, where $w\geq 0$ is $N-$harmonic and continous in $\overline{\Omega \cap B_R(p)} \setminus \{p \}$. If $w>0$ in $\Omega \cap B_R(p)$, by \cite{BBV,BoVe} we have that
\begin{equation} \label{2123}
\lim_{r \to 0^+}   r w(\sigma r+p)=-\langle \sigma,\nu(p)\rangle
\end{equation}
uniformly for $\sigma$ with $\langle \sigma,\nu(p)\rangle\leq -\delta<0$. Thanks to \eqref{2123}, as in Step 3 we still end up with the contradiction $\int_{\Omega \cap B_R(p)} e^w=+\infty$. Therefore, by the strong maximum principle we necessarily have that $w=0$ in $\Omega \cap B_R(p)$. Since $w_k$ is the part of $s_k$ which carries the information on the concentration phenomenon at $p$ and tends to disappear as $k \to +\infty$, we can expect that $s_k$ in the limit does not develop any singularities. We aim to show that $e^s \in L^q(\Omega \cap B_R(p))$ for all $q \geq 1$, by mimicking some arguments in \cite{3}. Letting $\varphi_k$ be the $N-$harmonic extension in $\Omega \cap B_R(p)$ of $s_k \mid_{\partial (\Omega \cap B_R(p))}$, for $M,a>0$ we have that
\begin{eqnarray} 
&&\int_{\Omega \cap B_R(p)} \langle |\nabla s_k|^{N-2}\nabla s_k -|\nabla w_k|^{N-2} \nabla w_k-|\nabla \varphi_k|^{N-2}\nabla \varphi_k, \nabla[ T_{M+a}(s_k-w_k-\varphi_k)-
 T_M (s_k-w_k-\varphi_k)]\rangle \nonumber \\ 
&&= \int_{\Omega \cap B_R(p)} (1-\chi_{\Omega \cap B_{r_k}(p)}) W_k e^{s_k} [ T_{M+a}(s_k-w_k-\varphi_k)-
 T_M (s_k-w_k-\varphi_k)] \nonumber \\
&&\leq  a \int_{\{|s_k-w_k-\varphi_k|>M \}} (1-\chi_{\Omega \cap B_{r_k}(p)}) W_k e^{s_k}, \label{1528}
\end{eqnarray}
%and
%\begin{eqnarray} 
%\int_{\Omega \cap B_R(p)} \langle |\nabla s_k|^{N-2}\nabla s_k -|\nabla \varphi_k|^{N-2}\nabla \varphi_k, \nabla T_M(s_k-\varphi_k) \rangle&=& \int_{\Omega \cap B_R(p)} W_k e^{s_k} %T_M(s_k-\varphi_k) \nonumber \\
%& \leq & M \sup_k \int_{\Omega \cap B_R(p)} V_k e^{u_k},  \label{0921}
%\end{eqnarray}
where the truncature operator $T_M$, $M>0$, is defined as 
$$T_M(u)=\left\{ \begin{array}{ll} -M & \hbox{if }u<-M \\ u &\hbox{if }|u|\leq M\\ M &\hbox{if }u>M. \end{array} \right.$$
The crucial property we will take advantage of is the following:
\begin{equation} \label{1527}
\sup_k \int_{\{|s_k-w_k-\varphi_k|>M \}} (1-\chi_{\Omega \cap B_{r_k}(p)}) W_k e^{s_k} \to 0 \qquad \hbox{as }M \to +\infty.
\end{equation}
Indeed, by \cite{34} notice that, up to a subsequence, we can assume that $\varphi_k \to \varphi$ in $C^1(\overline{\Omega \cap B_R(p)})$ as $k \to +\infty$, where $\varphi$ is the $N-$harmonic function in $\Omega \cap B_R(p)$ with $\varphi=s$ on $\partial (\Omega \cap B_R(p))$. Since $s_k-w_k-\varphi_k \to s-\varphi$ uniformly in $\Omega \cap (B_R(p) \setminus B_r(p))$ as $k \to +\infty$ for any given $r \in (0,R)$, we can find $M_r>0$ large so that 
$$\cup_k \{|s_k-w_k-\varphi_k|>M \} \subset \Omega \cap B_r(p)\qquad \forall\ M\geq M_r,$$
and then
$$\sup_k \int_{\{|s_k-w_k-\varphi_k|>M \}} (1-\chi_{\Omega \cap B_{r_k}(p)}) W_k e^{s_k} \leq \sup_k \int_{\Omega \cap B_r(p)} (1-\chi_{\Omega \cap B_{r_k}(p)}) W_k e^{s_k}$$
for all $M\geq M_r$. Since by \eqref{1339} and \eqref{0939} 
$$\int_{\Omega \cap B_r(p)} (1-\chi_{\Omega \cap B_{r_k}(p)}) W_k e^{s_k}\to \int_{\Omega \cap B_r(p)} W e^s$$
as $k \to +\infty$ and $We^s \in L^1(\Omega)$, for all $\epsilon>0$ we can find $r_\epsilon>0$ small so that
$$\sup_k \int_{\Omega \cap B_{r_\epsilon}(p)} (1-\chi_{\Omega \cap B_{r_k}(p)}) W_k e^{s_k} \leq \epsilon,$$
yielding to the validity of \eqref{1527}. Inserting \eqref{1527} into \eqref{1528} we get that, for all $\epsilon>0$, there exists $M_\epsilon$ so that
\begin{equation} \label{1529}
\int_{\{M< |s_k-w_k-\varphi_k|\leq M+a\} } \langle |\nabla s_k|^{N-2}\nabla s_k -|\nabla w_k|^{N-2} \nabla w_k-|\nabla \varphi_k|^{N-2}\nabla \varphi_k, \nabla(s_k-w_k-\varphi_k)\rangle \leq a \epsilon
\end{equation}
for all $M\geq M_\epsilon$ and $a>0$. Recall that $w_k \to 0$, $s_k \to s$ in $C^{1}_{loc}(\overline{\Omega \cap B_R(p)} \setminus \{p \} )$ and in $W^{1,q}(\Omega \cap B_R(p))$ for all $q<N$ as $k \to +\infty$  in view of \cite{8,9}.  Since 
$$\langle |\nabla s_k|^{N-2}\nabla s_k -|\nabla w_k|^{N-2} \nabla w_k, \nabla(s_k-w_k)\rangle \geq 0$$
and $\nabla \varphi_k$ behaves well, we can let $k \to +\infty$ in \eqref{1529} and by the Fatou Lemma get
\begin{equation} \label{1618}
\frac{d_N}{a} \int_{\{M< |s-\varphi|\leq M+a\} } |\nabla (s-\varphi)|^N \leq \frac{1}{a} \int_{\{M< |s-\varphi|\leq M+a\} } \langle |\nabla s|^{N-2}\nabla s -|\nabla \varphi|^{N-2}\nabla \varphi, \nabla(s-\varphi)\rangle \leq \epsilon
\end{equation}
for some $d_N>0$ and all $M\geq M_\epsilon$. 
%Similarly, by letting $k \to +\infty$ in \eqref{0921} we deduce that
%\begin{eqnarray} 
%d_N \int_{\Omega \cap B_R(p)} |\nabla T_M(s-\varphi)|^N \leq \int_{\{|s-\varphi| \leq M \} } \langle |\nabla s|^{N-2}\nabla s -|\nabla \varphi|^{N-2}\nabla \varphi, \nabla (s-\varphi) %\rangle \leq M \sup_k \int_{\Omega \cap B_R(p)} V_k e^{u_k}<+\infty  \label{1003}
%\end{eqnarray}
%yielding to $T_M(s) \in W^{1,N}(\Omega \cap B_R(p))$ near any $p \in S$ and then $T_M(s) \in W^{1,N}_0(\Omega)$ for all $M>0$. 
Introducing $H_{M,a}(s)=\frac{T_{M+a}(s-\varphi)-T_M(s-\varphi)}{a}$ and the distribution $\Phi_{s-\varphi}(M)= |\{x \in \Omega \cap B_R(p):|s-\varphi|(x)>M\}$ of $|s-\varphi|$, we have that
\begin{eqnarray*}
\Phi_{s-\varphi}(M+a)^{\frac{N-1}{N}} &\leq& \left(\int_{\Omega \cap B_R(p)} |H_{M,a}(s)|^{\frac{N}{N-1}}\right)^{\frac{N-1}{N}}\leq (N^N \omega_N)^{-\frac{1}{N}} 
\int_{\Omega \cap B_R(p)} |\nabla H_{M,a}(s)|\\
& \leq& (N^N \omega_N)^{-\frac{1}{N}} \frac{1}{a} \int_{\{M< |s-\varphi| \leq M+a\} } |\nabla (s-\varphi)|
\end{eqnarray*}
in view of the Sobolev embedding $W^{1,1}_0(\Omega \cap B_R(p)) \hookrightarrow L^{\frac{N}{N-1}}(\Omega \cap B_R(p))$ with sharp constant $(N^N \omega_N)^{-\frac{1}{N}}$, see \cite{FeFl}. By the H\"older inequality and \eqref{1618} we then deduce that
$$\Phi_{s-\varphi}(M+a) \leq (\frac{N^N d_N \omega_N }{\epsilon} )^{-\frac{1}{N-1}} \frac{\Phi_{s-\varphi}(M)-\Phi_{s-\varphi}(M+a)}{a}$$
for all $M \geq M_\epsilon$. By letting $a \to 0^+$ it follows that
$$\Phi_{s-\varphi}(M) \leq -(\frac{N^N d_N \omega_N}{\epsilon} )^{-\frac{1}{N-1}} \Phi_{s-\varphi}'(M)$$
for a.e. $M\geq M_\epsilon$, and by integration in $(M_\epsilon,M)$ 
$$\Phi_{s-\varphi}(M)\leq |\Omega \cap B_R(p)| \exp\left[- (\frac{N^N d_N \omega_N}{\epsilon})^{\frac{1}{N-1}} M\right]$$
for all $M\geq M_\epsilon$, in view of $\Phi_{s-\varphi}(M_\epsilon)\leq |\Omega \cap B_R(p)|$. Given $q \geq 1$ we can argue as follows:
\begin{eqnarray*} && \int_{\Omega\cap B_R(p)}e^{q |s-\varphi|}-|\Omega \cap B_R(p)|=q \int_{\Omega \cap B_R(p)} dx \int_0^{|s(x)-\varphi(x)|} e^{q M} dM=q \int_0^{\infty} e^{q M} \Phi_{s-\varphi}(M) dM\\
&& \leq |\Omega \cap B_R(p)| \left[ e^{q M_\epsilon}+q \int_{M_\epsilon}^{\infty} \exp \left( (q-(\frac{N^N d_N \omega_N}{\epsilon})^{\frac{1}{N-1}}) M \right) \right] dM<+\infty
\end{eqnarray*}
by taking $\epsilon$ sufficiently small. Since $\varphi \in C^1(\overline{\Omega \cap B_R(p)})$, we get that $e^s$ is a $L^q-$function near any $p \in S$, and then $e^s \in L^q(\Omega)$ for all $q\geq 1$. By the uniqueness result in \cite{DHM} and by Theorems \ref{T5}, \ref{T3}  we get that $s \in C^{1,\alpha}(\overline{\Omega})$, for some $\alpha \in (0,1)$.
\begin{oss}  The proof of $s \in C^{1,\alpha}(\overline{\Omega})$, $\alpha \in (0,1)$,  might be carried over in a shorter way. Indeed, the function $We^s \in L^1(\Omega)$ can be approximated either in a strong $L^1-$sense or in a weak measure-sense. In the former case, the limiting function $z$ is an entropy solution of
$$\left\{ \begin{array}{ll} -\Delta_N z=We^s &\hbox{in }\Omega\\
z=0&\hbox{on }\partial \Omega, \end{array} \right.$$
while in the latter we end up with $s$ by choosing $W_k e^{s_k}$ as the approximation in measure-sense. As consequence of the impressive uniqueness result in \cite{DHM}, $s=z$ and then $s$ is a entropy solution of \eqref{1417} (see \cite{3,7} for the definition of entropy solution). Lemma \ref{T4} is proved in \cite{3} for entropy solutions, and has been used there, among other things, to show that a entropy solution $s$ of \eqref{1417} is necessarily in $C^{1,\alpha}(\overline{\Omega})$, for some $\alpha \in (0,1)$. We have preferred a longer proof to give a self-contained argument which does not require to introduce special notions of distributional solutions (like SOLA, entropy and renormalized solutions, just to quote some of them).
\end{oss}

\medskip \noindent Fix any $p_{0}\in\partial\Omega$ and take $R_0>0$ small so that $\overline{B_{R_0}(p_{0})} \cap S= \{p_{0}\}$. Setting $y_{k}=p_{0}+\rho_{k,R} \nu(p_{0})$ with $0<R\leq R_0$ and
$$\rho_{k,R}=\frac{\displaystyle\int_{\partial\Omega\cap B_R(p_{0})}\langle x-p_{0},\nu \rangle|\nabla u_{k}|^{N}}{\displaystyle\int_{\partial\Omega\cap B_{R}(p_{0})}\langle\nu(p_{0}),\nu \rangle|\nabla u_{k}|^{N}},$$
we have that
\begin{equation}\label{E19}
\int_{\partial\Omega\cap B_R(p_{0})}\langle x-y_{k}, \nu \rangle|\nabla u_{k}|^{N}=0.
\end{equation}
Up to take $R_0$ smaller, we can assume that $|\rho_{k,R}|\leq 2R$. Applying Lemma \ref{Po} to $s_{k}$ on $\Omega\cap B_R(p_{0})$ with $y=y_{k}$, we obtain that
\begin{eqnarray}
\int_{\Omega\cap B_R(p_{0})}[NW_{k}+\langle x-y_{k},\nabla W_{k}\rangle]e^{s_{k}}&=&\int_{\partial(\Omega\cap B_R(p_{0}))}W_{k}e^{s_{k}}\langle x-y_{k},\nu\rangle \label{E21}\\
&&+\int_{\partial(\Omega\cap B_R(p_{0}))}\bigg[|\nabla s_{k}|^{N-2}\langle x-y_{k},\nabla s_{k}\rangle \partial_\nu s_{k}-\frac{|\nabla s_{k}|^{N}}{N}\langle x-y_{k},\nu\rangle\bigg]. \nonumber
\end{eqnarray}
We would like to let $k\rightarrow+\infty$, but $\partial(\Omega\cap B_R(p_{0}))$ contains the portion $\partial\Omega\cap B_R(p_{0})$ where the convergence $s_{k} \to s$ might fail. The clever choice of $\rho_{k,R}$, as illustrated by \eqref{E19}, leads to
\begin{eqnarray*}
\int_{\partial\Omega\cap B_R(p_{0})}\bigg[|\nabla s_{k}|^{N-2}\langle x-y_{k},\nabla s_{k}\rangle \partial_\nu s_{k} -\frac{|\nabla s_{k}|^{N}}{N}\langle x-y_{k},\nu\rangle\bigg]=(1-\frac{1}{N})\int_{\partial\Omega\cap B_R(p_{0})}|\nabla u_{k}|^{N}\langle x-y_{k},\nu \rangle=0
\end{eqnarray*}
in view of $\nabla s_{k}=\nabla u_{k}$ and $\nabla s_{k}=-|\nabla s_{k}| \nu$ on $\partial\Omega$ by means of $s_{k}=0$ on $\partial\Omega$. Hence, \eqref{E21} reduces to
\begin{eqnarray}\label{E22}
N \int_{\Omega\cap B_R(p_{0})}V_{k}e^{u_{k}}&=&-\int_{\Omega\cap B_R(p_{0})} \langle x-y_{k}, \frac{\nabla V_{k}}{V_{k}}\rangle V_{k}e^{u_{k}}
+\int_{\partial(\Omega\cap B_R(p_{0}))}V_{k} e^{u_k}\langle x-y_{k}, \nu\rangle\\
&&+\int_{\Omega\cap \partial B_R(p_{0})}\bigg[|\nabla s_{k}|^{N-2}\langle x-y_{k},\nabla s_{k}\rangle \partial_\nu s_{k} -\frac{|\nabla s_{k}|^{N}}{N}\langle x-y_{k},\nu\rangle\bigg].\nonumber
\end{eqnarray}
Since $|x-y_k|\leq 3R$ and $|\frac{\nabla V_{k}}{V_{k}}|\leq C_0^2$ in $\Omega\cap B_R(p_{0})$ in view of \eqref{E15}, by letting $k\rightarrow+\infty$ in \eqref{E22} we get that
$$N \mu\left(\Omega\cap B_R(p_0) \right) \leq 3R C_0^2 \mu\left(\Omega\cap B_R(p_0) \right) +3 C_0 R e^M| \partial(\Omega\cap B_R(p_{0}))| 
+ 3 R (1+\frac{1}{N})\int_{ \Omega\cap \partial B_R(p_0)} |\nabla s|^{N}$$
in view of $s_{k}\to s$ in $C^{1}_{loc}(\overline{\Omega}\setminus S )$. Since $s \in C^{1}(\overline{\Omega})$, by letting $R \to 0$ we deduce that $\mu(\{p_0\})=0$, and then $p_0 \notin \Sigma=S$. Since this is true for all $p_0 \in \partial \Omega$, we have shown that $S\subset \Omega$, and the proof of Step 5 is complete.

\medskip \noindent The combination of the previous 5 Steps provides us with a complete proof of Theorem \ref{T1}.
\begin{flushright}
$\Box$
\end{flushright}
\noindent Once Theorem \ref{T1} has been established, we can derive the following:\\
\emph{Proof (of Corollary \ref{T11}).}\\
By contradiction, assume the existence of sequences $\lambda_k \in \Lambda$, $V_k$ satisfying \eqref{E15} and $u_k \in C^{1,\alpha}(\overline{\Omega})$, $\alpha \in (0,1)$, weak solutions to \eqref{E1} so that $\|u_k\|_{\infty} \to +\infty$ as $k \to +\infty$. First of all, we can assume $\lambda_k >0$ (otherwise $u_k=0$) and 
\begin{equation} \label{2201}
\max_\Omega V_k e^{u_k-\alpha_k} \to +\infty
\end{equation}
as $k \to +\infty$ in view of Corollary \ref{2002}, where $\alpha_{k}=\log (\frac{\int_{\Omega}V_{k}e^{u_{k}}}{\lambda_{k}})$. The function $\hat{u}_{k}=u_{k}-\alpha_{k}$ solves
$$\left\{ \begin{array}{ll}
          -\Delta_{N}\hat{u}_{k}=V_{k}e^{\hat{u}_{k}}& \mbox{in}\ \Omega,\\
          \hat{u}_{k}=-\alpha_{k}& \mbox{on}\ \partial\Omega.\
        \end{array} \right.$$
Since $\lambda_{k}\in \Lambda$ and $\Lambda$ is a compact set, we have that $\sup_k \int_{\Omega}V_{k}e^{\hat{u}_{k}}=\sup_k \lambda_{k}<+\infty$, and then $\sup_k \int_{\Omega}e^{\hat{u}_{k}}<+\infty$ in view of \eqref{E15}. Since $\mbox{osc}_{\partial\Omega}(\hat{u}_{k})=0$, we can apply Theorem \ref{T1} to $\hat{u}_{k}$. Since $\max_\Omega \hat{u}_k \to +\infty$ as $k \to +\infty$ in view of \eqref{E15} and \eqref{2201}, alternative (iii) in Theorem \ref{T1} occurs for $\hat{u}_{k}$. By \eqref{1919} we get that
$$\lambda_k=\int_{\Omega}V_{k}e^{\hat{u}_{k}} \to c_N m$$
as $k \to +\infty$, for some $m \in \mathbb{N}$. Hence, $c_N m\in \Lambda$, in contradiction with $\Lambda \subset [0,+\infty) \setminus c_N \mathbb{N}$. 
\begin{flushright}
$\Box$
\end{flushright}

\section{A general existence result}\label{Sec5}
\noindent The Moser-Trudinger inequality \cite{zz} states that, for some $C_\Omega>0$, there holds
\begin{equation}\label{E28}
\int_{\Omega}\exp(\alpha|u|^{\frac{N}{N-1}})dx\leq C_\Omega
\end{equation}
for all $u\in W_{0}^{1,N}(\Omega)$ with $\|u\|_{W^{1,N}_0(\Omega)}\leq1$ and all $\alpha\leq\alpha_{N}=(N^N \omega_{N})^{\frac{1}{N-1}}$, whereas \eqref{E28} is false when $\alpha>\alpha_{N}$. A simple consequence of \eqref{E28}, always referred to as the Moser-Trudinger inequality, is the following:
\begin{equation}\label{E29}
\log\bigg(\int_{\Omega}e^{u}dx\bigg)\leq \frac{1}{Nc_{N}}\|u\|_{W_{0}^{1,N}(\Omega)}^{N}
+\log C_{\Omega}\end{equation}
for all $u\in W_{0}^{1,N}(\Omega)$, where $c_{N}$ is defined in Theorem \ref{T1}. Indeed, \eqref{E29} follows by \eqref{E28} by noticing
$$u \leq [(\frac{N\alpha_N}{N-1})^{-\frac{N-1}{N}} \|u\|_{W^{1,N}_0(\Omega)}]\times  [(\frac{N\alpha_N}{N-1})^{\frac{N-1}{N}}\frac{|u|}{\|u\|_{W^{1,N}_0(\Omega)}}] \leq
\frac{1}{N c_N} \|u\|^N_{W^{1,N}_0(\Omega)} 
+\alpha_N |\frac{u}{\|u\|_{W^{1,N}_0(\Omega)}}|^{\frac{N}{N-1}}$$
in view of the Young's inequality. By \eqref{E29} it follows that:
 $$ J_{\lambda}(u)\geq\frac{1}{N}(1-\frac{\lambda}{c_{N}})\|u\|_{W_{0}^{1,N}(\Omega)}^{N}-  \lambda \log (C_0 C_{\Omega})
$$
for all $u\in W_{0}^{1,N}(\Omega)$ in view of \eqref{E15}, where $J_\lambda$ is given in \eqref{E27}. Hence, $J_{\lambda}$ is bounded from below for $\lambda\leq c_{N}$ and coercive for $\lambda<c_N$. Since  the map $ u \in W_{0}^{1,N}(\Omega)\to Ve^u \in L^{1}(\Omega)$ is compact in view of \eqref{E29} and the embedding $W_{0}^{1,N}(\Omega)\hookrightarrow L^{2}(\Omega)$ is compact, for $\lambda<c_{N}$ we have that $J_{\lambda}$ attains the global minimum in $W_{0}^{1,N}(\Omega)$, and then \eqref{E1} is solvable. In Theorem \ref{thm0954} we just consider the difficult case $\lambda>c_N$. Notice that a solution $u \in W^{1,N}_0(\Omega)$ of \eqref{E1} belongs to $C^{1,\alpha}(\overline{\Omega})$ for some $\alpha \in (0,1)$, in view of \eqref{E29} and Theorems \ref{T5}, \ref{T3}.

\medskip \noindent The constant $\frac{1}{Nc_{N}}$ in \eqref{E29} is optimal as it follows by evaluating the inequality along 
$$U(\frac{x-p}{\epsilon})-\frac{N^2}{N-1}\log \epsilon\,,\quad p \in \Omega,$$
as $\epsilon \to 0$, up to make a cut-off away from $p$ so to have a function in $W^{1,N}_0(\Omega)$. The function $U$ is given in \eqref{E16} and, as already mentioned in the Introduction,  satisfies
$$\int_{\mathbb{R}^N} e^U=c_N.$$
Indeed, the equation $-\Delta_N U=e^U$ does hold pointwise in $\mathbb{R}^N \setminus \{0\}$, and then can be integrated in $B_R(0)\setminus B_\epsilon(0)$, $0<\epsilon<R$, to get
$$\int_{B_R(0)\setminus B_\epsilon(0) } e^U=-\int_{\partial B_R(0)} |\nabla U|^{N-2} \langle \nabla U, \nu \rangle+\int_{\partial B_\epsilon(0)} |\nabla U|^{N-2} \langle \nabla U, \nu \rangle,$$
where $\nu(x)=\frac{x}{|x|}$. Letting $\epsilon \to 0$ and $R \to +\infty$, we get that
$$\int_{\mathbb{R}^N} e^U=N(\frac{N^2}{N-1})^{N-1} \omega_N=c_N$$
in view of
$$\nabla U=-\frac{N^{2}}{N-1}\frac{|x|^{\frac{N}{N-1}-2}x}{1+|x|^{\frac{N}{N-1}}}.$$
Since $\frac{1}{Nc_{N}}$ in \eqref{E29} is optimal, the functional $J_{\lambda}$ is unbounded from below for $\lambda>c_{N}$, and our goal is to develop a global variational strategy to find a critical point of saddle type. The classical Morse theory states that a sublevel is a deformation retract of an higher sublevel unless there are critical points in between, and the crucial assumption on the functional is the validity of the so-called Palais-Smale condition.
Unfortunately, in our context such assumption fails since $J_\lambda$ admits unbounded Palais-Smale sequences for $\lambda\geq c_{N}$, see \cite{aha1, aha}. This technical difficulty can be overcome by using a method introduced by Struwe that exploits the monotonicity of the functional $\frac{J_{\lambda}}{\lambda}$ in $\lambda$. An alternative approach has been found in \cite{aha}, which provides a deformation between two sublevels unless $J_{\lambda_k}$ has critical points in the energy strip for some sequence $\lambda_k \to \lambda$.
Thanks to the compactness result in Corollary \ref{T11} and the a-priori estimates in Theorem \ref{T3}, we have at hands the following crucial tool:
\begin{lm}\label{T15}
Let $\lambda \in (c_N,+\infty) \setminus c_N \mathbb{N}$. If $J_{\lambda}$ has no critical levels $u$ with $a \leq J_\lambda(u) \leq b$, then $J_{\lambda}^{a}$ is a deformation retract of $J_{\lambda}^{b}$, where
$$J_\lambda^t=\{u \in W^{1,N}_0(\Omega):\, J_\lambda(u)\leq t \}.$$
\end{lm}
\noindent To attack existence issues for \eqref{E1} when $\lambda \in (c_N,+\infty) \setminus c_N \mathbb{N}$, it is enough to find any two sublevels $J_\lambda^a$ and $J_\lambda^b$ which are not homotopically equivalent.

\medskip \noindent Hereafter, the parameter $\lambda$ is fixed in $(c_N,+\infty) \setminus c_N \mathbb{N}$. By Corollary \ref{T11} and Theorem \ref{T3} we have that $J_\lambda$ does not have critical points with large energy. Exactly as in \cite{013},  Lemma \ref{T15} can be used to construct a deformation retract of $W_{0}^{1,N}(\Omega)$ onto very high sublevels of $J_{\lambda}$. More precisely, we have the following
\begin{lm}\label{T17}
There exists $L>0$ large so that $J_{\lambda}^{L}$ is a deformation retract of $W_{0}^{1,N}(\Omega)$. In particular, $J_{\lambda}^{L}$ is contractible.\end{lm}
\noindent For the sake of completeness, we give some details of the proof.
\begin{proof}
Take $L \in\mathbb{N}$ large so that $J_{\lambda}$ has no critical points $u$ with $J_{\lambda}(u)\geq L$. By Lemma \ref{T15} $J_{\lambda}^{n}$ is a deformation retract of $J_{\lambda}^{n+1}$ for all $n\geq L$,  and $\eta_{n}$ will denote the corresponding retraction map. Given $u \in W_{0}^{1,N}(\Omega)$ with $J_{\lambda}(u)>L$, by setting recursively 
$$
\left\{ \begin{array}{l}
         \eta^{1,n}(s,u)=\eta_{n}(s,u)\\
         \eta^{2,n}(s,u)=\eta_{n-1}(s-1,\eta_{n}(1,u))\\
         \vdots\\
         \eta^{k+1,n}=\eta_{n-k}(s-k,\eta^{(k)}(k,u)),\
        \end{array} \right.
\vspace{0,2cm}$$
for $s\geq 0$ we consider the following map
$$
\hat{\eta}(s,u)=\left\{ \begin{array}{ll}
         \eta^{k+1,n}(s,u)& \mbox{if}\ n< J_{\lambda}(u)\leq n+1 \hbox{ for }n\geq L, s\in[k,k+1]\\
         u& \mbox{if}\ J_{\lambda}(u)\leq L.\
        \end{array} \right.
\vspace{0,2cm}$$
Next, define $s_u$ as the first $s>0$ such that $J_{\lambda}(\hat{\eta}(s,u))=L$ if $J_{\lambda}(u)>L$ and as $0$ if $J_{\lambda}(u)\leq L.$ The map $\eta(t,u)=\hat{\eta}(t s_u,u):[0,1] \times W^{1,N}_0(\Omega) \to W^{1,N}_0(\Omega)$ satisfies
$\eta(1,u) \in   J_{\lambda}^L$ for $u \in W^{1,N}_0(\Omega)$ and $\eta(t,u)=u$ for $(t,u)\in [0,1] \times J_{\lambda}^L$. Since $s_u$ depends continuously in $u$, the map $\eta$ is continuous in both variables, providing us with the required deformation retract.
\end{proof}
\medskip \noindent Thanks to Lemmas \ref{T15} and \ref{T17}, we are led to study the topology of sublevels for $J_\lambda$ with very low energy. The real core of such a global variational approach is an improved form \cite{0100} of the Moser-Trudinger inequality for functions $u \in W^{1,N}_0(\Omega)$ with a measure $\frac{Ve^u}{\int_\Omega V e^u}$ concentrated on several subomains in $\Omega$. As a consequence, when $\lambda\in(c_{N}m, c_{N}(m+1))$ and $J_{\lambda}(u)$ is very negative, the measure $\frac{Ve^{u}}{\int_{\Omega}Ve^{u}}$ can be concentrated near at most $m$ points of $\overline{\Omega}$, and can be naturally associated to an element $\sigma \in \mathcal{B}_{m}(\overline{\Omega})$, where
$$\mathfrak{B}_{m}(\overline{\Omega}):=\{ \sum_{i=1}^{m} t_{i}\delta_{p_i}:\ t_{i}\geq0,\ \sum_{i=1}^{m} t_{i}=1, \ p_i \in\overline{\Omega}\}$$
has been first introduced by Bahri and Coron in \cite{501,502} and is known in literature as the \emph{space of formal barycenters} of $\overline{\Omega}$ with order $m$. The topological structure of $J_{\lambda}^{-L}$, $L>0$ large, is completely characterized in terms of $\mathcal{B}_{m}(\overline{\Omega})$. The non-contractibility of $\mathcal{B}_{m}(\overline{\Omega})$ let us see a change in topology between $J_{\lambda}^L$ and $J_\lambda^{-L}$ for $L>0$ large, and by Lemma \ref{T15} we obtain the existence result claimed in Theorem \ref{thm0954}. Notice that our approach is simpler than the one in \cite{DJLW1,012,0120} (see also \cite{2000}), by using \cite{aha} instead of the Struwe's monotonicity trick to bypass the general failure of PS-condition for $J_{\lambda}$.

\medskip \noindent As already explained, the key point is the following improvement of the Moser-Trudinger inequality:
\begin{lm}\label{T18}
Let $\Omega_{i}$, $i=1,\dots,l+1$, be subsets of $\overline{\Omega}$ so that $\mbox{dist}(\Omega_{i},\Omega_{j})\geq\delta_{0}>0$, for $i\neq j$, and $\gamma_{0}\in(0,\frac{1}{l+1})$. Then, for any $\epsilon>0$ there exists a constant $C=C(\epsilon, \delta_{0}, \gamma_{0})$ such that there holds
$$\log (\int_{\Omega}V e^{u}dx)\leq \frac{1}{Nc_{N}(l+1-\epsilon)}\|u\|_{W_{0}^{1,N}(\Omega)}^{N}+C$$
for all $u\in W_{0}^{1,N}(\Omega)$ with
\begin{equation}\label{E31}
\frac{\int_{\Omega_{i}} V e^{u}}{\int_{\Omega}V e^{u}}\geq\gamma_{0}\quad i=1,\dots,l+1.
\end{equation}
\end{lm}
\begin{proof} Let $g_1,\dots,g_{l+1} $ be cut-off functions so that $0\leq g_i \leq 1$, $g_{i}=1$ in $\Omega_{i}$, $g_{i}=0$ in $\{\hbox{dist}(x,\Omega_{i})\geq\frac{\delta_{0}}{4} \}$ and $\|g_{i}\|_{C^{2}(\overline{\Omega})}\leq C_{\delta_0}.$ Since $g_i$, $i=1,\dots,l$, have disjoint supports, for all $u\in W_{0}^{1,N}(\Omega)$ there exists $i=1,\dots,l+1$ such that
\begin{equation}\label{E32}
\int_{\Omega}(g_i |\nabla u|)^{N}\leq\frac{1}{l+1}\int_{\displaystyle \cup_{i=1}^{l+1}\mbox{supp} g_i}|\nabla u|^{N}\leq\frac{1}{l+1}\|u\|_{W_{0}^{1,N}(\Omega)}^{N}.
\end{equation}
Since by the Young's inequality
\begin{eqnarray*}
|\nabla(g_{i}u)|^{N} &\leq& (g_{i}|\nabla u|+|\nabla g_{i}||u|)^{N} \leq (g_{i}|\nabla u|)^{N}+C_1\big[(g_i|\nabla u|)^{N-1}|\nabla g_{i}||u|+(|\nabla g_{i}||u|)^{N}\big]\\
& \leq&  [1+ \frac{\epsilon}{(l+1)(3l+3-\epsilon)}](g_{i}|\nabla u|)^{N}+C_2(|\nabla g_{i}| |u|)^{N}
\end{eqnarray*}
for all $\epsilon>0$ and some $C_1>0$, $C_2=C_2(\epsilon)>0$, we have that
$$\|g_{i}u\|_{W_{0}^{1,N}(\Omega)}^{N} \leq \int_{\Omega}(g_{i}|\nabla u|)^{N}+\frac{\epsilon}{(l+1)(3l+3-\epsilon)}  \|u\|_{W_{0}^{1,N}(\Omega)}+N c_N C_3 \|u\|_{L^{N}(\Omega)}^{N},$$
where $C_3=\frac{ C_2 C_{\delta_{0}}^{N}}{N c_N}$. Since $g_i u \in W^{1,N}_0(\Omega)$, by \eqref{E29} and \eqref{E32} it follows that
\begin{equation}\label{E33}
\int_{\Omega}e^{g_{i}u}\leq C_{\Omega} \exp\bigg(\frac{3}{Nc_{N}(3l+3-\epsilon )}\|u\|_{W_{0}^{1,N}(\Omega)}^{N}+C_3 \|u\|_{L^{N}(\Omega)}^{N}\bigg)
\end{equation}
does hold for all $u \in W^{1,N}_0(\Omega)$ and some $i=1,\dots,l+1$. 

\medskip \noindent Let $\eta \in (0, |\Omega|)$ be given. Since $\{|u| \geq 0\}=\Omega$ and $\displaystyle \lim_{a\to +\infty} |\{| u| \geq a\}|=0$, the set
$$A_\eta=\{ a \geq 0:\ |\{ |u| \geq a\}|\geq \eta\}$$
is non-empty and bounded from above. Letting $a_\eta=\sup A_\eta$, we have that $a_\eta \geq 0$ is a finite number so that
\begin{equation} \label{1837}
\ |\{ |u| \geq a_\eta\}|\geq \eta,\quad \ |\{ |u| \geq a\}|< \eta\quad \forall \ a>a_\eta
\end{equation}
in view of the left-continuity of the map $a \to \ |\{ |u| \geq a\}|$. Given $\eta>0$ and $u \in W^{1,N}_0(\Omega)$ satisfying \eqref{E31}, we can fix $a=a_\eta$ and $i=1,\dots,l+1$
so that \eqref{E33} applies to $(|u|-2a)_{+}$ yielding to
\begin{eqnarray*}
\int_{\Omega} V e^{u} \leq \frac{1}{\gamma_{0}}\int_{\Omega_{i}} Ve^{|u|}\leq\frac{C_0 e^{2a}}{\gamma_{0}}\int_{\Omega}e^{g_{i}(|u|-2a)_{+}} \leq \frac{C_0 C_\Omega }{\gamma_{0}} \exp\bigg(\frac{3}{Nc_{N}(3l+3-\epsilon)}\|u\|_{W_{0}^{1,N}(\Omega)}^{N}+2a+C_3 \|(|u|-2a)_{+}\|_{L^{N}(\Omega)}^{N}\bigg)
\end{eqnarray*}
in view of \eqref{E15}. By the Poincar\'{e} and Young inequalities and the first property in \eqref{1837} it follows that
$$2a \leq \frac{2}{\eta} \int_{\{|u|\geq a\}}|u| \leq \frac{C_5}{\eta}\|u\|_{W_{0}^{1,N}(\Omega)} \leq \frac{3 \epsilon}{Nc_{N}(3l+3-\epsilon)(3l+3-2\epsilon)} \|u\|_{W_{0}^{1,N}(\Omega)}^{N}+C_6$$
for some $C_5>0$ and  $C_6=C_6(\epsilon,\eta)>0$, and there holds 
$$\|(|u|-2a)_{+}\|_{L^{N}(\Omega)}^{N}\leq \eta^{\frac{1}{2}}\|(|u|-2a)_{+}\|_{L^{2N}(\Omega)}^{N}\leq C_4 \eta^{\frac{1}{2}}\|u\|_{W_{0}^{1,N}(\Omega)}^{N}$$
for some $C_4>0$ in view of the H\"{o}lder and Sobolev inequalities and the second property in \eqref{1837}. Choosing $\eta$ small as
$$\eta = \left(\frac{\epsilon}{C_3 C_4  N c_N (3l+3-2\epsilon)(l+1-\epsilon)}\right)^2,$$
we finally get that
$$\int_{\Omega}V e^{u}\leq\frac{C_0 C_{\Omega}}{\gamma_{0}}\exp\bigg(\frac{1}{Nc_{N}(l+1-\epsilon)}\|u\|_{W_{0}^{1,N}(\Omega)}^{N}+C \bigg)$$
for some $C=C(\epsilon, \delta_{0}, \gamma_{0})$.
\end{proof}
\noindent A criterium for the occurrence of \eqref{E31} is the following:
\begin{lm}\label{T19}
Let $l\in\mathbb{N}$ and $0<\epsilon,r<1$. There exist $\bar{\epsilon}>0$ and $\bar{r}>0$ such that, for every $0\leq f\in L^{1}(\Omega)$ with 
\begin{equation} \label{1550}
\|f\|_{L^{1}(\Omega)}=1\ , \quad \int_{ \Omega \cap \bigcup_{i=1}^l B_r (p_i) } f <1-\epsilon\qquad \forall\: p_1,\dots,p_l \in \overline{\Omega},
\end{equation}
there exist $l+1$ points $\bar{p}_1,\dots,\bar{p}_{l+1} \in\overline{\Omega}$ so that
$$\int_{\Omega \cap B_{\bar{r}}(\bar{p}_i)}f \geq\bar{\epsilon}\  ,\qquad B_{2\bar{r}}(\bar{p}_i) \cap B_{2\bar{r}}(\bar{p}_j)=\emptyset\quad \forall \ i \neq j.$$
\end{lm}
\begin{proof} By contradiction, for all $\bar{\epsilon},\bar{r}>0$ we can find $0\leq f\in L^{1}(\Omega)$ satisfying \eqref{1550} such that, for every $(l+1)$-tuple of points $p_{1},...,p_{l+1}\in\overline{\Omega}$ the statement
\begin{equation}\label{E34}
\int_{\Omega \cap B_{\bar{r}}(p_{i})}f\geq\bar{\epsilon}\ , \qquad B_{2\bar{r}}(p_{i})\cap B_{2\bar{r}}(p_{j})=\emptyset\quad \forall \ i\neq j
\end{equation}
is false. Setting $\bar{r}=\frac{r}{8}$, by compactness we can find $h$ points $x_{i}\in\overline{\Omega}$, $i=1,\dots,h$, such that $\overline{\Omega} \subset \bigcup_{i=1}^{h} B_{\bar{r}}(x_{i})$. Setting $\bar{\epsilon}=\frac{\epsilon}{2h}$, there exists $i=1,...,h$ such that $\int_{\Omega \cap B_{\bar{r}}(x_{i})}f\geq\bar{\epsilon}$. Let $\{\tilde{x}_{1},...,\tilde{x}_{j}\}\subseteq\{x_{1},...,x_{h}\}$ be the maximal set with respect to the property $\int_{\Omega \cap B_{\bar{r}}(\tilde{x}_{i})}f\geq\bar{\epsilon}.$ Set $j_{1}=1$ and let $X_{1}$ denote the set
 $$X_{1}= \Omega \cap \bigcup_{i\in\Lambda_{1}}B_{\bar{r}}(\tilde{x}_{i}) \subseteq \Omega \cap B_{6\bar{r}}(\tilde{x}_{j_{1}}),\quad \Lambda_{1}=\{i=1,...,j:\  B_{2\bar{r}}(\tilde{x}_{i})\cap B_{2\bar{r}}(\tilde{x}_{j_{1}})\neq\emptyset\}.$$
If non empty, choose $j_{2}\in\{1,...,j\}\setminus\Lambda_{1}$, i.e. $B_{2\bar{r}}(\tilde{x}_{j_{2}})\cap B_{2\bar{r}}(\tilde{x}_{j_{1}})=\emptyset.$ Let $X_{2}$ denote the set
 $$X_{2}=\Omega \cap \bigcup_{i\in\Lambda_{2}}B_{\bar{r}}(\tilde{x}_{i})\subseteq\Omega \cap B_{6\bar{r}}(\tilde{x}_{j_{2}}), \quad \Lambda_{2}=\{i=1,...,j:\ B_{2\bar{r}}(\tilde{x}_{i})\cap B_{2\bar{r}}(\tilde{x}_{j_{2}})\neq\emptyset\}.$$
Iterating this process, if non empty, at the $l-$th step we choose $j_l\in\{1,...,j\}\setminus\bigcup_{j=1}^{l-1}\Lambda_{j}$, i.e. $B_{2\bar{r}}(\tilde{x}_{j_l}) \cap B_{2\bar{r}}(\tilde{x}_{j_i})=\emptyset$ for all $i=1,\dots,l-1$, and we define 
$$X_l=\Omega \cap \bigcup_{i\in\Lambda_l}B_{\bar{r}}(\tilde{x}_{i})\subseteq \Omega \cap B_{6\bar{r}}(\tilde{x}_{j_l}), \quad \Lambda_l=\{i=1,...,j:\ B_{2\bar{r}}(\tilde{x}_{i})\cap B_{2\bar{r}}(\tilde{x}_{j_l})\neq\emptyset\}.$$
By \eqref{E34} the process has to stop at the $s-$th step with $s\leq l$. By the definition of $\bar{r}$ we obtain
 $$\Omega \cap \bigcup_{i=1}^{j} B_{\bar{r}}(\tilde{x}_{i}) \subset\bigcup_{i=1}^{s} X_{i}\subset \Omega \cap \bigcup_{i=1}^{s}B_{6\bar{r}}(\tilde{x}_{j_{i}})\subset \Omega \cap \bigcup_{i=1}^{s} B_{r}(\tilde{x}_{j_{i}})$$
in view of $\{1,...,j\}=\bigcup_{i=1}^{s}\Lambda_{i}$.
Therefore, we have that
\begin{eqnarray*}
\int_{ \Omega \setminus\bigcup_{i=1}^{s} B_{r}(\tilde{x}_{j_{i}})}f  \leq \int_{\Omega \setminus\bigcup_{i=1}^{j} B_{\bar{r}}(\tilde{x}_{i})}f=\int_{( \Omega \cap \bigcup_{i=1}^{h} B_{\bar{r}}(x_{i}))\setminus (\bigcup_{i=1}^{j} B_{\bar{r}}(\tilde{x}_{i}))}f<(h-j)\bar{\epsilon}<\frac{\epsilon}{2}
\end{eqnarray*}
in view of the definition of $\tilde{x}_{1}, \dots ,\tilde{x}_{j}$. Define $p_i$ as $\tilde{x}_{j_{i}}$ for $i=1,\dots,s$ and as $\tilde{x}_{j_{s}}$ for $i=s+1,\dots,l$. Since $\int_{\Omega \setminus\bigcup_{i=1}^{l}B_{r}(p_{i})}f<\frac{\epsilon}{2}$, we deduce that
$$\int_{\Omega \cap \bigcup_{i=1}^{l} B_{r}(p_{i})}f=\int_{\Omega}f-\int_{\Omega \setminus \bigcup_{i=1}^{l} B_{r}(p_{i})}f>1-\frac{\epsilon}{2}>1-\epsilon,$$
contradicting the second property in \eqref{1550}. The proof is complete.
\end{proof} 
\noindent As a consequence, we get that
\begin{lm}\label{T20}
Let $\lambda\in\big(c_{N}m, c_{N}(m+1)\big)$, $m\in\mathbb{N}$. For any $0<\epsilon,r <1$ there exists a large $L=L(\epsilon,r)>0$ such that, for every $u\in W_{0}^{1,N}(\Omega)$ with $J_{\lambda}(u)\leq-L$, we can find $m$ points $p_{i,u}\in\overline{\Omega}$, $ i=1,\dots,m$, satisfying
$$\int_{\Omega\setminus \cup_{i=1}^{m} B_{r}(p_{i,u})}Ve^{u} \leq \epsilon\int_{\Omega}Ve^{u}.$$
\end{lm}
\begin{proof} By contradiction there exist $\epsilon,\ r \in (0,1)$ and functions $u_{k} \in W_{0}^{1,N}(\Omega)$ so that $J_{\lambda}(u_{k})\to-\infty$ as $k\to+\infty$ and
\begin{equation}\label{E36}
\int_{\Omega\setminus \cup_{i=1}^{m} B_{r}(p_{i})}Ve^{\hat{u}_{k}} >\epsilon
\end{equation}
for all $p_{1},...,p_{m} \in \overline{\Omega}$, where $\hat{u}_{k}=u_{k}-\log\int_{\Omega}Ve^{u_{k}}$.
Since
$$\int_{\Omega\setminus \cup_{i=1}^{m} B_{r}(p_{i})}Ve^{\hat{u}_{k}}=
\int_{\Omega} Ve^{\hat{u}_{k}} -\int_{\Omega \cap \cup_{i=1}^{m} B_{r}(p_{i})}Ve^{\hat{u}_{k}}
=1-\int_{\Omega \cap \cup_{i=1}^{m} B_{r}(p_{i})}Ve^{\hat{u}_{k}},$$
by \eqref{E36} we get that
$$\int_{\Omega \cap \cup_{i=1}^{m}B _{r}(p_{i})}Ve^{\hat{u}_{k}}< 1-\epsilon$$ 
for all $m$-tuple $p_{1},\dots,p_{m} \in \overline{\Omega}$. Applying Lemma \ref{T19} with $l=m$ and $f=Ve^{\hat{u}_{k}}$, we find $\bar \epsilon,\bar r>0$ and $\bar{p}_{1},\dots,\bar{p}_{m+1}\in\overline{\Omega}$ so that
$$\int_{\Omega \cap B_{\bar{r}}(\bar{p}_{i})} Ve^{u_{k}} \geq  \bar{\epsilon} \int_{\Omega}Ve^{u_{k}},\qquad B_{2\bar{r}}(\bar{p}_{i})\cap B_{2\bar{r}}(\bar{p}_{j})=\emptyset \quad \forall\ i\neq j.$$
Applying Lemma \ref{T18} with $\Omega_{i}=\Omega \cap B_{\bar{r}}(\bar{p}_{i})$ for $i=1,\dots,m+1$, $\delta_{0}=2\bar{r}$ and $\gamma_{0}=\bar{\epsilon}$, it now follows that
$$\log\bigg(\int_{\Omega}Ve^{u_{k}} \bigg)\leq\frac{1}{Nc_{N}(m+1-\eta)} \|u\|_{W_{0}^{1,N}(\Omega)}^{N}+C$$ for all $\eta>0$, for some $C=C(\eta,\delta_{0},\gamma_{0},a,b)$.
Since $\lambda<c_{N}(m+1)$, we get that
$$J_{\lambda}(u_{k})= \frac{1}{N}\|u_{k}\|_{W_{0}^{1,N}(\Omega)}^{N}-\lambda\log\bigg(\int_{\Omega}Ve^{u_{k}}dx\bigg)\geq
\frac{1}{N} \left(1-\frac{\lambda}{c_{N}(m+1-\eta)}\right)\|u_{k}\|_{W_{0}^{1,N}(\Omega)}^{N}-C\lambda \geq -C \lambda$$
for $\eta>0$ small, in contradiction with $J_{\lambda}(u_{k})\to -\infty$ as $k \to +\infty$.
\end{proof}
\noindent The set $\mathcal{M}(\overline{\Omega})$ of all Radon measures on $\overline{\Omega}$ is a metric space with the Kantorovich-Rubinstein distance, which is induced by the norm 
$$\|\mu\|_{\ast}=\sup_{\|\phi \|_{Lip(\overline{\Omega})} \leq 1}  \int_{\Omega} \phi d\mu,\qquad \mu\in\mathcal{M}(\overline{\Omega}).$$ 
Lemma \ref{T20} can be re-phrased as
\begin{lm} \label{1122}
Let $\lambda\in\big(c_{N}m, c_{N}(m+1)\big)$, $m\in\mathbb{N}$. For any $\epsilon>0$ small there exists a large $L=L(\varepsilon)>0$ such that, for every $u\in W_{0}^{1,N}(\Omega)$ with $J_{\lambda}(u)\leq-L$, we have 
\begin{equation} \label{1050}
\mbox{dist}\bigg(\frac{Ve^{u}}{\int_{\Omega}Ve^{u}},\mathfrak{B}_{m}(\overline{\Omega})\bigg)\leq\epsilon.
\end{equation}
\end{lm}
\begin{proof} Given $\epsilon\in (0,2)$ and $r=\frac{\epsilon}{4}$, let $L=L(\frac{\epsilon}{4},r)>0$ be as given in Lemma \ref{T20}. For all $u\in W_{0}^{1,N}(\Omega)$ with $J_{\lambda}(u)\leq-L$, let us denote for simplicity as $p_{1},\dots,p_{m}\in\overline{\Omega}$ the corresponding points $p_{1,u},\dots ,p_{n,u}$ such that
\begin{equation}\label{E37}
\int_{\Omega\setminus\bigcup_{i=1}^{m}B_{r}(p_{i})}Ve^{u}\leq \frac{\epsilon}{4} \int_{\Omega}Ve^{u}.
\end{equation}
Define $\sigma\in\mathfrak{B}_{m}(\overline{\Omega})$ as 
$$\sigma=\sum_{i=1}^{m}t_{i}\delta_{p_{i}},\qquad  t_{i}=\frac{ \int_{A_{r,i}}Ve^{u}}{\int_{\Omega \cap \bigcup_{i=1}^{m}B_{r}(p_{i})}Ve^{u}},$$ 
where $A_{r,i}=(\Omega \cap B_{r}(p_{i})) \setminus\bigcup_{j=1}^{i-1}B_{r}(p_j)$. Since $A_{r,i}$, $i=1,\dots,m$, are disjoint sets with $\bigcup_{i=1}^{m}A_{r,i}=\Omega \cap \bigcup_{i=1}^{m}B_{r}(p_i)$, we have that $\sum_{i=1}^{m}t_{i}=1$ and
\begin{eqnarray*}
\bigg|\int_\Omega \phi \left[ Ve^u dx-(\int_\Omega V e^u) d\sigma \right] \bigg|&\leq& \bigg|\int_{\Omega\setminus\bigcup_{i=1}^{m}B_{r}(p_i)} Ve^{u}\phi \bigg|+\bigg|\int_{\Omega \cap \bigcup_{i=1}^{m}B_{r}(p_i)}Ve^{u}\phi-(\int_\Omega V e^u)\sum_{i=1}^{m}t_{i}\phi(p_i)\bigg|\\
&\leq & \frac{\epsilon}{4} \int_\Omega V e^u+\sum_{i=1}^{m} \bigg|\int_{A_{r,i}}Ve^{u}\phi-(\int_\Omega Ve^u) t_{i}\phi(p_i)\bigg|\\
&\leq& \frac{\epsilon}{4}\int_\Omega V e^u+\sum_{i=1}^{m} \int_{A_{r,i}}Ve^{u}|\phi-\phi(p_i)|
+ \bigg|\frac{\int_\Omega Ve^u}{\int_{\Omega \cap \bigcup_{i=1}^{m}B_{r}(p_i)}Ve^{u}}-1\bigg|\sum_{i=1}^{m}\int_{A_{r,i}}Ve^{u}\\
&\leq & \left(\frac{\epsilon}{4}+r +\frac{\epsilon}{4-\epsilon}\right) \int_\Omega V e^u
\end{eqnarray*}
in view of \eqref{E37}, $\|\phi\|_{Lip(\overline{\Omega})}\leq1$ and
$$\bigg|\frac{\int_\Omega V e^u}{\int_{\Omega \cap \bigcup_{i=1}^{m}B_{r}(p_i)}Ve^{u}}-1\bigg| \leq \frac{\epsilon}{4-\epsilon}.$$ 
Since there holds
$$\bigg|\int_\Omega \phi \left[ \frac{Ve^u dx}{\int_\Omega V e^u} -d\sigma \right] \bigg| \leq \epsilon$$
for all $\phi \in Lip(\overline{\Omega})$ with $\|\phi\|_{Lip(\overline{\Omega})}\leq1$, we have that
$$\| \frac{Ve^u}{\int_\Omega V e^u} -\sigma \|_\ast \leq \epsilon$$
for some $\sigma \in \mathfrak{B}_{m}(\overline{\Omega})$, and then 
$$\mbox{dist}\bigg(\frac{Ve^{u}}{\int_{\Omega}Ve^{u}},\mathfrak{B}_{m}(\overline{\Omega})\bigg)\leq\epsilon.$$
The proof is complete.
\end{proof}
\noindent When \eqref{1050} does hold, one would like to project $\frac{V e^u}{\int_\Omega V e^u}$ onto $\mathfrak{B}_{m}(\overline{\Omega})$. To avoid boundary points (which cause troubles in the construction of the map $\Phi$ below) we replace $\overline{\Omega}$ by its retract of deformation $K=\{x\in\Omega:\ \mbox{dist}(x,\partial\Omega)\geq \delta\}$, $\delta>0$ small. Since $\mathfrak{B}_{m}(K)$ is a retract of deformation of $\mathfrak{B}_{m}(\overline{\Omega})$, by \cite{BBB} there exists a projection map
$$\Pi_{m}:\{\sigma\in\mathcal{M}(\overline{\Omega}):\ dist(\sigma,\mathfrak{B}_{m}(\overline{\Omega}))<\epsilon_{0}\}\ \to \ \mathfrak{B}_{m}(K),\quad \epsilon_0>0 \hbox{ small},$$
which is continuous with respect to the Kantorovich-Rubinstein distance. Thanks to $\Pi_m$ and Lemma \ref{1122}, for $\epsilon \leq \epsilon_0$ there exist $L=L(\epsilon)>0$ large and a continuous map
$$\begin{array}{cccc}\Psi: & J_\lambda^{-L} &\to & \mathfrak{B}_{m}(K)\\
& u & \to & \Pi_m(\frac{Ve^u}{\int_\Omega Ve^u}). \end{array}$$
The key point now is to construct a continuous map $\Phi:\mathfrak{B}_{m}(K)\rightarrow J_{\lambda}^{-L}$ so that $\Psi \circ \Phi$ is homotopically equivalent to $\hbox{Id}_{\mathfrak{B}_{m}(K)}$. When $\mathfrak{B}_{m}(\overline{\Omega})$ is non contractible, the same is true for $\mathfrak{B}_{m}(K)$ and then for $J_\lambda^{-L}$ for $L>0$ large. Theorem \ref{thm0954} then follows by Lemmas \ref{T15} and \ref{T17}.

\medskip \noindent The construction of $\Phi$ relies on an appropriate choice of a one-parameter family of functions $\varphi_{\epsilon,\sigma}$, $\sigma\in\mathfrak{B}_{m}(K)$, modeled on the standard bubbles $U_{\epsilon,p}$, see \eqref{E16}. Letting $\chi \in C_{0}^{\infty}(\Omega)$ be so that $\chi=1$ in $\Omega_{\frac{\delta}{2}}=\{x\in\Omega:\ \mbox{dist}(x,\partial\Omega)>\frac{\delta}{2} \}$,  we define 
$$\varphi_{\epsilon,\sigma}(x)=\chi(x) \log\sum_{i=1}^{m}t_{i}\bigg(\frac{F_{N}}{(\epsilon^{\frac{N}{N-1}}+|x-p_{i}|^{\frac{N}{N-1}})^{N}V(p_{i})}\bigg),$$
where $\sigma=\displaystyle \sum_{i=1}^m t_i \delta_{p_i} \in\mathfrak{B}_{m}(K)$ and $\epsilon>0$. Since $\varphi_{\epsilon,\sigma}\in W_{0}^{1,N}(\Omega)$, the map $\Phi$ can be constructed as $\Phi_{\epsilon_0}$, $\epsilon_0>0$ small, where
$$\begin{array}{cccc}\Phi_\epsilon: &\mathfrak{B}_{m}(K) &\to &  J_\lambda^{-L}\\
& \sigma & \to & \varphi_{\epsilon,\sigma}. \end{array}$$
To map $\mathfrak{B}_{m}(K)$ into the very low sublevel $J_\lambda^{-L}$, the difficult point is to produce uniform estimates in $\sigma$ as $\epsilon \to 0$. We have  
\begin{lm}\label{T27}
There hold
\begin{enumerate}
\item there exist $C_{0}>0$ and $\epsilon_{0}>0$ so that
$$\bigg\|\frac{Ve^{\varphi_{\epsilon,\sigma}}}{\int_{\Omega}Ve^{\varphi_{\epsilon,\sigma}}}-\sigma\bigg\|_{\ast}\leq C_{0}\epsilon$$
for all $0<\epsilon\leq\epsilon_{0}$ and $\sigma \in \mathfrak{B}_{m}(K)$;
\item $J_{\lambda}(\varphi_{\epsilon,\sigma})\rightarrow-\infty$  as $\epsilon\rightarrow0$ uniformly in $\sigma\in\mathfrak{B}_{m}(K)$.
\end{enumerate}
\end{lm}
\begin{proof}  Recall that 
$$U_{\epsilon,p}(x)=\log \bigg(\frac{F_{N}\epsilon^{\frac{N}{N-1}}}{(\epsilon^{\frac{N}{N-1}}+|x-p|^{\frac{N}{N-1}})^N} \bigg).$$
Fix $\phi \in Lip(\overline{\Omega})$ with $\|\phi\|_{Lip(\overline{\Omega})}\leq1$. Since $\varphi_{\epsilon,\sigma}$ is bounded from above in $\Omega \setminus \Omega_{\frac{\delta}{2}}$ uniformly in $\sigma$, we have that
\begin{eqnarray} \label{1419}
\int_{\Omega}Ve^{\varphi_{\epsilon,\sigma}} \phi&=&\epsilon^{-\frac{N}{N-1}}  
\sum_{i=1}^{m} \int_{\Omega_{\frac{\delta}{2}}}  \frac{t_i V \phi }{V(p_i)} e^{U_{\epsilon,p_i}}  +O(1) =\epsilon^{-\frac{N}{N-1}} \sum_{i=1}^m  \int_{B_{\frac{\delta}{2}}(p_{i})} \frac{t_i V \phi}{V(p_i)}  e^{U_{\epsilon,p_i}} +O(1)  \\
&=&\epsilon^{-\frac{N}{N-1}}\left( c_N \int_\Omega \phi d \sigma+O(\epsilon)\right)
\nonumber
\end{eqnarray}
as $\epsilon \to 0$ uniformly in $\phi$ and $\sigma$. We have used that
$$\int_{B_{\frac{\delta}{2}}(p_i)} \frac{V \phi}{V(p_i)} e^{U_{\epsilon,p_i}}=\int_{B_{\frac{\delta}{2 \epsilon}}(0)}  (\phi(p_i)+O(\epsilon |y|)) e^U= c_N \phi(p_i)+O(\epsilon)$$
does hold as $\epsilon \to 0$, uniformly in $\phi$ and $\sigma$, in view of \eqref{E15}. Therefore, there holds
$$\bigg| \int_\Omega \phi \left(\frac{Ve^{\varphi_{\epsilon,\sigma}}}{\int_{\Omega}Ve^{\varphi_{\epsilon,\sigma}}}dx-d\sigma\right) \bigg|\leq C_0 \epsilon$$
for all $\phi \in Lip(\overline{\Omega})$ with $\|\phi\|_{Lip(\overline{\Omega})}\leq1$, and then
$$\|\frac{Ve^{\varphi_{\epsilon,\sigma}}}{\int_{\Omega}Ve^{\varphi_{\epsilon,\sigma}}}-\sigma\|_\ast \leq C_0 \epsilon$$
for all $\sigma \in \mathfrak{B}_{m}(K)$. Part (1) is proved.

\medskip \noindent For part (2), it is enough to show that
\begin{eqnarray}\label{E36qq}
&& \log\displaystyle\int_{\Omega}Ve^{\varphi_{\epsilon,\sigma}}=\frac{N}{N-1}\log\frac{1}{\epsilon}+O(1)\\
\label{E37qq}
&& \frac{1}{N}\displaystyle\int_{\Omega}|\nabla\varphi_{\epsilon,\sigma}|^{N} \leq \frac{N}{N-1}c_{N}m\log\frac{1}{\epsilon}+O(1)
\end{eqnarray}
as $\epsilon\rightarrow0$ uniformly in $\sigma\in\mathfrak{B}_{m}(K)$, in view of $\lambda>mc_{N}$. Estimate \eqref{E36qq} follows by \eqref{1419} with $\phi=1$. As far as \eqref{E37qq} is concerned, let us set $\varphi_{\epsilon,\sigma}=\chi \tilde{\varphi}_{\epsilon,\sigma}$. All the estimates below are uniform in $\sigma$. Since
$$\nabla \tilde{\varphi}_{\epsilon,\sigma}=- \frac{N^2}{N-1} \frac{\sum_{i=1}^{m}t_{i}V(p_{i})^{-1}(\epsilon^{\frac{N}{N-1}}+|x-p_{i}|^{\frac{N}{N-1}})^{-(N+1)}|x-p_{i}|^{\frac{N}{N-1}-2}(x-p_{i})}{\sum_{i=1}^{m}t_{i}V(p_{i})^{-1}(\epsilon^{\frac{N}{N-1}}+|x-p_{i}|^{\frac{N}{N-1}})^{-N}},$$
we have that $\|\tilde{\varphi}_{\epsilon,\sigma}\|_{C^1(\Omega\setminus\Omega_{\frac{\delta}{2}})}=O(1)$ and then
$$|\nabla\varphi_{\epsilon,\sigma}|=O(1)$$ 
in $\Omega\setminus\Omega_{\frac{\delta}{2}}$. Therefore we can write that
\begin{equation} \label{pp1}
\frac{1}{N}\int_{\Omega}|\nabla\varphi_{\epsilon,\sigma}|^{N}=\frac{1}{N}\int_{\Omega_{\frac{\delta}{2}}} |\nabla\tilde{\varphi}_{\epsilon,\sigma}|^{N}+O(1).
\end{equation}
We estimate $|\nabla\tilde{\varphi}_{\epsilon,\sigma}|$ in two different ways:\\
(i) $|\nabla\tilde{\varphi}_{\epsilon,\sigma}|(x)\leq\frac{N^{2}}{N-1}\frac{1}{d(x)},$ where $d(x)=\min \{|x-p_{i}|:, i=1,...,m\}$;\\
(ii) $|\nabla\tilde{\varphi}_{\epsilon,\sigma}|\leq\frac{N^{2}}{N-1}C_{0}\epsilon^{-1}$ in view of
 $$\frac{\epsilon|x-p_{i}|^{\frac{N}{N-1}-1}}{\epsilon^{\frac{N}{N-1}}+|x-p_{i}|^{\frac{N}{N-1}}}\leq C_{0}$$
by the Young's inequality. By estimate (ii) we have that
\begin{equation} \label{pp2}
\int_{\Omega_{\frac{\delta}{2}}}|\nabla\tilde{\varphi}_{\epsilon,\sigma}|^{N}=\int_{\Omega_{\frac{\delta}{2}}\setminus\bigcup_{j=1}^{m}B_{\epsilon}(p_{j})}|\nabla\tilde{\varphi}_{\epsilon,\sigma}|^{N}+O(1)\leq\sum_{j=1}^{m}\int_{A_{j}\setminus B_{\epsilon}(p_{j})}|\nabla\tilde{\varphi}_{\epsilon,\sigma}|^{N}+O(1)
\end{equation}
in view of $\Omega_{\frac{\delta}{2}}\setminus\bigcup_{j=1}^{m}B_{\epsilon}(p_{j})\subset\bigcup_{j=1}^{m}\bigg(A_{j}\setminus B_{\epsilon}(p_{j})\bigg)$, where $A_j=\{x\in\Omega_{\frac{\delta}{2}}:\ |x-p_j|=d(x) \}$. 
Since by estimate (i) we have that
$$ \int_{A_{j}\setminus B_{\epsilon}(p_{j})}|\nabla\tilde{\varphi}_{\epsilon,\sigma}|^{N} \leq (\frac{N^{2}}{N-1})^{N}  \int_{A_{j}\setminus B_{\epsilon}(p_{j})}\frac{1}{|x-p_{j}|^{N}} \leq (\frac{N^{2}}{N-1})^{N} \int_{B_R(0)\setminus B_{\epsilon}(0)}\frac{1}{|x|^{N}}+O(1)=\frac{N^{2}}{N-1}c_{N}\log\frac{1}{\epsilon}+O(1)$$
in terms of $R=\hbox{diam}\ \Omega$, by \eqref{pp1}-\eqref{pp2} we deduce the validity of \eqref{E37qq}. The proof is complete.
\end{proof}

\medskip \noindent In order to prove that $\Psi\circ\Phi$ is homotopically equivalent to $\hbox{Id}_{\mathfrak{B}_{m}(K)}$, we construct an explicit homotopy $H$ as follows
$$H:(0,1]\longrightarrow C\big((\mathfrak{B}_{m}(K), \|\cdot\|_{\ast});(\mathfrak{B}_{m}(K),\|\cdot\|_{\ast})\big),\ t\mapsto H(t)=\Psi\circ\Phi_{t\varepsilon_{0}}.$$
The map $H$ is continuous in $(0,1]$ with respect to the norm $\|\cdot\|_{\infty,\mathfrak{B}_{m}(K)}$.  In order to conclude, we need to prove that there holds
$$\lim_{t \to 0} \|H(t)-\mbox{Id}_{\mathfrak{B}_{m}(K)}\|_{\infty,\mathfrak{B}_{m}(K)}=\lim_{\epsilon \to 0} \sup_{\sigma\in\mathfrak{B}_{m}(K) } \|\Psi\circ\Phi_{\epsilon}(\sigma)-\sigma\|_{\ast}=0,$$
where $\epsilon=t\epsilon_{0}$. Since $\Pi_m(\sigma)=\sigma$ and $\mathfrak{B}_{m}(K)$ is a compact set in $\left(\mathcal{M}(\overline{\Omega}), \|\cdot \|_{\ast} \right)$, by the continuity of $\Pi_{m}$ in $\|\cdot\|_\ast$ and Lemma \ref{T27}-(1) we deduce that
$$\|\Psi\circ\Phi_{\epsilon}(\sigma)-\sigma\|_{\ast}
=\|\Pi_m \left( \frac{Ve^{\varphi_{\epsilon,\sigma}}}{\int_{\Omega}Ve^{\varphi_{\epsilon,\sigma}}} \right)-\Pi_m(\sigma)\|_{\ast} \to 0$$
as $\epsilon \to 0$, uniformly in $\sigma\in\mathfrak{B}_{m}(K)$. Finally, we extend $H(t)$ at $t=0$ in a continuous way by setting $H(0)=id_{\mathfrak{B}_{m}(K)}$.

\medskip \noindent Let us now discuss the main assumption in Theorem \ref{thm0954}. In \cite{007} it is claimed that $\mathfrak{B}_{m}(\Omega)$ is non contractible for all $m\geq1$ if $\Omega$ is non contractible too, as it arises for closed manifolds \cite{0120}. However, by the techniques in \cite{050} it is shown in \cite{kal} that $\mathfrak{B}_{m}(X)$ is contractible for all $m \geq 1$, for a non contractible topological and acyclic (i.e. with trivial $\mathbb{Z}-$homology) space $X$. A concrete example is represented by the punctured Poincar\'{e} sphere, and it is enough to take a tubular neighborhood $\Omega$ of it to find a counterexample to the claim in \cite{007}. A sufficient condition for the main assumption in Theorem \ref{thm0954} is the following:
\begin{thm} \cite{kal} \label{T22}
Assume that $X$ is homotopically equivalent to a finite simplicial complex. Then $\mathfrak{B}_{m}(X)$ is non contractible for all $m\geq2$ if and only if $X$ is not acyclic (i.e. with non trivial $\mathbb{Z}$-homology).
\end{thm}

{\section*{Appendix}
\setcounter{thm}{0}
\setcounter{equation}{0}
\setcounter{subsection}{0}
\renewcommand{\thesection}{A}
\noindent Let us collect here some useful regularity estimates which have been frequently used throughout the paper. Concerning $L^\infty-$estimates, the general interior estimates in \cite{41} are used here to derive also boundary estimates for solutions $u \in W^{1,N}_c(\Omega)=\{u \in W^{1,N}(\Omega):\ u\mid_{\partial \Omega}=c \}$, $c \in \mathbb{R}$, through the \emph{Schwarz reflection principle}. 

\medskip \noindent Given $x_{0}\in\partial\Omega$, we can find a smooth diffeomorphism $\psi$ from a small ball $B \subset \mathbb{R}^N$, $0\in B$, into a neighborhood $V$ of $x_0$ in $\mathbb{R}^{N}$ so that $\psi(B \cap\{y_{N}=0\})=V\cap\partial\Omega$ and $\psi(B^+)=V\cap \Omega$, where $B^+=B \cap\{y_{N}>0\}$.
Letting $u_0\in W_{c}^{1,N}(\Omega)$ be a critical point of 
$$\frac{1}{p}\int_{\Omega}|\nabla u|^N-\int_{\Omega}f u,\quad u \in W_{c}^{1,N}(\Omega),$$ 
then $v_0=u_0 \circ \psi$ is a critical point of
$$I(v)=\int_{B^{+}}\left[\frac{1}{N}| A(y)\nabla v|^N-fv \right]|\det\nabla\psi|,\quad v \in \mathcal{V},$$
in view of $|\nabla u|^N \circ \psi=|A \nabla v|^N$ in $B^+$ for $v=u\circ \psi$, where $ A(y)=(D\psi^{-1})^t (\psi(y))$ is an invertible $N\times N$ matrix for all $y \in B^+$ and 
$$\mathcal{V}=\{v\in W^{1,N}(B^{+}): v=c \ \mbox{on}\ y_{N}=0\ \mbox{and}\ v=u_0\circ \psi \ \mbox{on } \partial B \cap\{y_N>0\} \}.$$
In the sequel, $g_{\sharp}$ and $g^{\sharp}$ denote the odd and even extension in $B$ of a function $g$ defined on $B^+$, respectively. Decomposing the matrix $A$ as
$$A=
\left(
\begin{array}{c|c}
A' & a_1 \\ \hline a_2 & a_{NN}
\end{array}
\right)
$$
with $a_1,a_2:B^+ \to \mathbb{R}^{N-1}$, for $y \in B$ let us introduce  
$$A^\sharp=
\left(
\begin{array}{c|c}
(A')^{\sharp} & (a_1)_{\sharp} \\ \hline (a_2)_{\sharp} & (a_{NN})^{\sharp}
\end{array}
\right).$$
The odd reflection $(v_0-c)_{\sharp}+c \in W^{1,N}(B)$ is a weak solution in $B$ of
$$-\mbox{div}\  \mathcal{A}(y,\nabla v)=(f |\det\nabla\psi|)_{\sharp},$$
where $\mathcal{A}:\, (y,p) \in B \times \mathbb{R}^N \to |\det\nabla\psi|^\sharp |A^\sharp(y) p|^{N-2} [(A^\sharp)^t A^\sharp](y) p \in \mathbb{R}^N.$ In view of the invertibility of $A(y)$ for all $y\in B^+$, the map $\mathcal{A}$ satisfies 
\begin{equation}\label{E11qq}
|\mathcal{A}(y,p)|\leq a|p|^{N-1},\quad \langle p, \mathcal{A}(y,p) \rangle \geq a^{-1}|p|^{N}\end{equation}
for all $y \in B$ and $p \in \mathbb{R}^N$, for some $a>0$. Since $2c-u\leq u$ when $u\geq c$, thanks to \eqref{E11qq} we can now apply the general local interior estimates of J. Serrin in \cite{41} to get:
\begin{thm}\label{T5}
Let $u\in W_{loc}^{1,N}(\Omega)$ be a weak solution of 
\begin{equation} \label{0857}
-\Delta_N u=f \quad \mbox{in}\ \Omega.
\end{equation}
Assume that $f\in L^{\frac{N}{N-\epsilon}}(\Omega\cap B_{2R})$, $0<\epsilon\leq 1$, and $u \in W^{1,N}(\Omega \cap B_{2R})$ satisfies $u=c$ on $\partial \Omega \cap \overline{B_{2R}}$, $u \geq c$ in $\Omega \cap B_{2R}$ for some $c \in \mathbb{R}$ if $\partial \Omega \cap \overline{B_{2R}}\neq \emptyset$. Then, the following estimates do hold:
\begin{eqnarray*}
&&\|u^{+}\|_{L^{\infty}(\Omega\cap B_R)} \leq C (\|u^{+}\|_{L^N(\Omega\cap B_{2R})}+1)\\
&&\|u\|_{L^{\infty}(\Omega\cap B_R) } \leq C(\|u\|_{L^N(\Omega\cap B_{2R})}+1) \quad (\hbox{if }c=0)
\end{eqnarray*}
for some $C=C\left(N,a,\epsilon,R, \|f\|_{L^{\frac{N}{N-\epsilon}}(\Omega\cap B_{2R})} \right).$
\end{thm}
\noindent Since the Harnack inequality in \cite{41} is very general, it can be applied in particular when $\mathcal A$ satisfies \eqref{E11qq}, by allowing us to treat also boundary points through the \emph{Schwarz reflection principle}. The following statement is borrowed from \cite{pucser}:
\begin{thm}\label{T7}
Let $u\in W_{loc}^{1,N}(\Omega)$ be a nonnegative weak solution of \eqref{0857}, where $f\in L^{\frac{N}{N-\epsilon}}(\Omega)$, $0<\epsilon\leq 1$. Let $\Omega'\subset \Omega$ be a sub-domain of $\Omega$. Assume that $u \in W^{1,N}(\Omega \cap \Omega')$ satisfies $u=0$ on $\partial \Omega \cap \overline{\Omega'}$. Then, there exists $C=C(N,\epsilon,\Omega')$ so that
$$\sup_{\Omega'}u\leq C\left(\inf_{\Omega'}u+
\|f\|^{\frac{1}{N-1}}_{L^{\frac{N}{N-\epsilon}}(\Omega)} \right).$$
\end{thm}
\noindent By choosing $\Omega'=\Omega$ we deduce that
\begin{cor} \label{2002}
Let $u\in W^{1,N}_0(\Omega)$ be a weak solution of $-\Delta_{N} u=f$ in $\Omega$, where $f\in L^{\frac{N}{N-\epsilon}}(\Omega)$, $0<\epsilon\leq 1$. Then, there exists a constant $C=C(N,\epsilon,\Omega)$ such that 
$$\|u\|_{L^\infty(\Omega)}\leq C
\|f\|^{\frac{1}{N-1}}_{L^{\frac{N}{N-\epsilon}}(\Omega)}.$$
\end{cor}
\noindent Thanks to Theorem \ref{T5}, by the estimates in \cite{20,34,46} we now have that
\begin{thm}\label{T3}
Let $u\in W_{loc}^{1,N}(\Omega)$ be a weak solution of \eqref{0857}. Assume that $f \in L^\infty(\Omega\cap B_{2R})$, and $u \in W^{1,N}(\Omega \cap B_{2R})$ satisfies $u=0$ on $\partial \Omega \cap B_{2R}$. Then, there holds $\|u\|_{C^{1,\alpha}(\Omega\cap B_{R})}\leq C=C=C(N,a,R,\|f\|_{\infty,\Omega \cap B_{2R}},\|u\|_{L^N(\Omega\cap B_{2R})})$, for some $\alpha \in (0,1)$.
\end{thm}

\medskip \noindent We will now consider \eqref{0857} with a Dirac measure $\delta_{p_0}$ as R.H.S. In our situation, the fundamental solution $\Gamma$ takes the form
$$\Gamma(|x|)= (N\omega_{N})^{-\frac{1}{N-1}} \log \frac{1}{|x|}.$$
In a very general framework, Serrin has described in \cite{41} the behavior of solutions near a singularity. In particular, every $N$-harmonic and continuous function $u$ in $\Omega \setminus \{0\}$, which is bounded from below in $\Omega$, has either a removable singularity at $0$ or there holds
\begin{equation}\label{E6}
\frac{1}{C}\Gamma\leq u\leq C\Gamma
\end{equation}
in a neighborhood of $0$, for some $C\geq 1$. For the $p-$Laplace operator Kichenassamy and Veron \cite{27} have later improved \eqref{E6} by expressing $u$ in terms of $\Gamma$. A combination of \cite{27,41} leads in our situation to:
\begin{thm}\label{T2}
Let $u$ be a $N$-harmonic continuous function in $\Omega-\{0\}$, which is bounded from below in $\Omega$. Then there exists $\gamma \in \mathbb{R}$ such that
$$u-\gamma\Gamma\in L_{loc}^{\infty}(\Omega)$$
and $u$ is a distributional solution in $\Omega$ of 
$$-\Delta_N u=\gamma|\gamma|^{N-2}\delta_{0}$$
with $|\nabla u|^{N-1}\in L_{loc}^{1}(\Omega)$. Moreover, for $\gamma\neq0$ there holds
$$\lim_{x\rightarrow0}|x|^{|\alpha|}D^{|\alpha|}(u-\gamma\Gamma)(x)=0$$
for all multi-indices $\alpha=(\alpha_{1},...,\alpha_{N})$ with length $|\alpha|=\alpha_{1}+...+\alpha_{N}\geq1$.
\end{thm}

\bibliographystyle{plain}

\end{document}